\documentclass[11pt]{article}
\usepackage[T1]{fontenc}
\usepackage[utf8]{inputenc}
\usepackage{geometry}
\usepackage{amsmath}
\usepackage{amsthm}
\usepackage[dvipsnames]{xcolor} 
\usepackage[color=Purple!50!white, textwidth=22mm]{todonotes}
\usepackage[unicode=true,
 bookmarks=false,
 breaklinks=false,pdfborder={0 0 1},backref=section,colorlinks=false]
 {hyperref}

\usepackage{comment}

\makeatletter
\theoremstyle{plain}
\newtheorem{thm}{\protect\theoremname}[section]
\theoremstyle{definition}
\newtheorem{defn}[thm]{\protect\definitionname}
\theoremstyle{plain}
\newtheorem{lem}[thm]{\protect\lemmaname}
\theoremstyle{plain}
\newtheorem{claim}[thm]{\protect\claimname}
\theoremstyle{plain}
\newtheorem{cor}[thm]{\protect\corollaryname}
\theoremstyle{plain}

\usepackage{amssymb}

\makeatother

\providecommand{\corollaryname}{Corollary}
\providecommand{\definitionname}{Definition}
\providecommand{\lemmaname}{Lemma}
\providecommand{\claimname}{Claim}
\providecommand{\theoremname}{Theorem}

\title{Homogeneous structures in subset sums and non-averaging sets} 
\author{David Conlon\thanks{Department of Mathematics, California Institute of Technology, Pasadena, CA 91125. Email:  {\tt dconlon@caltech.edu.} Research supported by NSF Award DMS-2054452.} \and Jacob Fox\thanks{Department of Mathematics, Stanford University, Stanford, CA 94305. Email: {\tt jacobfox@stanford.edu}. Research supported by a Packard Fellowship and by NSF Awards DMS-1855635 and DMS-2154169.} \and Huy Tuan Pham\thanks{Department of Mathematics, Stanford University, Stanford, CA 94305. Email: {\tt huypham@stanford.edu}. Research supported in part by a Clay Research Fellowship and a Two Sigma Fellowship.}}
\date{}

\begin{document}

\maketitle

\begin{abstract}
We show that for every positive integer $k$ there are positive constants $C$ and $c$ such that if $A$ is a subset of $\{1, 2, \dots, n\}$ of size at least $C n^{1/k}$, then, for some $d \leq k-1$, the set of subset sums of $A$ contains a homogeneous $d$-dimensional generalized arithmetic progression of size at least $c|A|^{d+1}$.
This strengthens a result of Szemer\'edi and Vu, who proved a similar statement without the homogeneity condition. 
As an application, we make progress on the Erd\H{o}s--Straus non-averaging sets problem, showing that every subset $A$ of $\{1, 2, \dots, n\}$ of size at least $n^{\sqrt{2} - 1 + o(1)}$ contains an element which is the average of two or more other elements of $A$. This gives the first polynomial improvement on a result of Erd\H{o}s and S\'ark\"ozy from 1990.
\end{abstract}

\section{Introduction}

What is the largest subset $A$ of $[n] := \{1,2,\dots, n\}$ with the property that no element of $A$ is the average of two or more other elements of $A$? Such sets, known in the literature as non-averaging sets, were first introduced by Erd\H{o}s and Straus~\cite{Str,EStr} in the late 1960s. If we write $h(n)$ for the size of the largest non-averaging subset of $[n]$, then the bounds 
$$\Omega(n^{1/4}) \leq h(n) \leq n^{1/2 +o(1)},$$
with the lower bound due to Bosznay~\cite{Bosz} and the upper bound to Erd\H{o}s and S\'ark\"ozy~\cite{ES}, were both known by 1990. 
Bypassing a bottleneck which we shall elaborate on below, we give a polynomial improvement to the upper bound on $h(n)$, namely, $h(n)\leq n^{\sqrt{2} - 1 + o(1)}$.

The principal tool used in the proof of this result, and the  
main result of this paper, is a homogeneous strengthening of a seminal result of Szemer\'edi and Vu~\cite{SV2} about the existence of generalized arithmetic progressions in subset sums. Before saying more about non-averaging sets, let us describe this result in more detail.

\subsection{Homogeneous generalized arithmetic progressions in subset sums}

Given a set or a sequence $A$ of integers, the \emph{set of subset sums} $\Sigma(A)$ is the set of all integers representable as a sum of distinct elements from $A$. That is, 
\[\Sigma(A) = \left\{ \sum_{s \in S} s : S \subseteq A \right\}.\]

One of the fundamental results about subset sums is the following theorem of Szemer\'edi and Vu~\cite{SV}. 

\begin{thm}[Szemer\'edi--Vu \cite{SV}]\label{SVthm}
There is a constant $C$ such that if $A \subset [n]$ with $|A| \geq C\sqrt{n}$, then $\Sigma(A)$ contains an arithmetic progression of length $n$.
\end{thm} 

The bound in Theorem \ref{SVthm} is easily seen to be best possible up to the constant factor $C$ by considering, for example, the set of all positive integers up to $\lfloor \sqrt{2n} - 2\rfloor$. 
This theorem improved on earlier results of Freiman~\cite{Frei3} and S\'{a}rk\"{o}zy~\cite{Sar2}, who both showed that there is a constant $C$ such that if $|A|\ge C\sqrt{n\log n}$, then $\Sigma (A)$ contains an arithmetic progression of length at least $n$. However, it also loses something, because the Freiman--S\'ark\"ozy result gives not only an arithmetic progression, but a {\it homogeneous progression}, an arithmetic progression $a, a + d, \dots, a + kd$ where the common difference $d$ divides $a$ and, hence, every other term in the progression. The natural question, raised by several groups of authors~\cite{ES, Sar3, TVW}, of whether there is a common strengthening of the Szemer\'edi--Vu and Freiman--S\'ark\"ozy theorems was recently answered in the affirmative by the authors~\cite{CFP}.

\begin{thm}[Conlon--Fox--Pham~\cite{CFP}] \label{thm:Sze-Vu-2}
There is a constant $C$ such that if $A \subset [n]$ with $|A| \geq C\sqrt{n}$, then $\Sigma(A)$ contains a homogeneous progression of length $n$.
\end{thm}

Our concern in this paper will be with higher-dimensional analogues of this result. Recall that a \emph{generalized arithmetic progression} or \emph{GAP}, for short, is a set of the form 
$$Q=\{x+\sum_{i=1}^{d}n_{i}q_{i}:0\le n_{i}\le w_{i}-1\},$$
where $d, x, q_1, \dots, q_d, w_1, \dots, w_d$ are integers with $d$, the {\it dimension}, and all the $w_i$ positive. Here and throughout, we will implicitly assume that the $n_i$ are also integers. We say that the GAP is \emph{proper} if all sums in the definition are distinct or, equivalently, if $|Q| = w_1 w_2 \cdots w_d$. Generalizing Theorem~\ref{SVthm} above, Szemer\'edi and Vu~\cite{SV2} proved the following result.

\begin{thm}[Szemer\'edi--Vu \cite{SV2}]\label{SVthm2}
For every positive integer $k$, there are positive constants $C$ and $c$ such that if $A$ is a subset of $[n]$ of size $m \ge Cn^{1/k}$, then, for some $d\le k-1$, $\Sigma(A)$ contains a proper $d$-dimensional GAP $P$ of size at least $cm^{d+1}$. 
\end{thm}

For instance, when $k = 3$, this result says that there are positive constants $C$ and $c$ such that if $A \subset [n]$ with $|A| \geq C n^{1/3}$, then $\Sigma(A)$ contains either an arithmetic progression of length at least $c |A|^2$ or a $2$-dimensional GAP of size at least $c|A|^3$. Moreover, a construction of Szemer\'edi and Vu~\cite[Section 3]{SV2} shows that the theorem is essentially best possible.

Following the notation above, we say that a GAP is \emph{homogeneous} if $\gcd(q_{1},\dots,q_{d}) | x$, which clearly generalizes the definition of homogeneous for ordinary $1$-dimensional arithmetic progressions. In light of Theorem~\ref{thm:Sze-Vu-2}, it is natural to ask whether there is also a homogeneous version of Theorem~\ref{SVthm2}. The following result gives a positive answer to this question.

\begin{thm} \label{thm:hom-AP} 
For every positive integer $k$, there are positive constants $C$ and $c$ such that if $A$ is a subset of $[n]$ of size $m \ge Cn^{1/k}$, then, for some $d  \le k-1$, $\Sigma(A)$ contains a proper homogeneous $d$-dimensional GAP $P$ of size at least $cm^{d+1}$. 
\end{thm}

As well as being interesting in its own right, it was already pointed out by Nguyen and Vu~\cite{NV} that such a result can be considerably simpler to apply than its non-homogeneous counterpart. More than this, for some applications, homogeneity seems to be essential. Our result on non-averaging sets, which we discuss further in the next subsection, is such an example. 

The main step in proving Theorem \ref{thm:hom-AP} is to establish the
following intermediate result, which is already sufficient for our 
application to non-averaging sets. Given a homogeneous GAP $Q=\{\sum_{i=1}^{d}n_iq_i : a_i\le n_i \le b_i\}$ and a positive real number $c$, we let $cQ = \{\sum_{i=1}^{d} n_iq_i : ca_i\le n_i\le cb_i\}$. We say that $cQ$ is proper if all the sums in the definition are distinct. 

\begin{thm} \label{thm:hom-AP-build} 
For any $\beta>1$ and $0<\eta<1$, there are positive constants $c$ and $d$ such that the following holds. Let $A$ be a subset of $[n]$ of size $m$ with $n\le m^{\beta}$ and let $s\in[m^{\eta},cm/\log m]$. Then there exists a subset $\hat{A}$ of $A$ of size at least $m-c^{-1}s\log m$, a proper GAP $P$ of dimension at most $d$ such
that $\hat{A} \cup \{0\}$ is contained in $P$ and a subset
$A'$ of $\hat{A}$ of size at most $s$ such that $\Sigma(A')$
contains a homogeneous translate of $csP$, where $csP$ is proper. 
\end{thm}

This result is clearly tight up to the constant $c$, since, if $\hat{A} \cup \{0\}$ is contained in $P$, then $\Sigma(A')$ is contained in $sP$ for any subset $A'$ of $\hat{A}$ of size $s$. This almost tight relationship between the homogeneous GAP that we find in $\Sigma(A')$ and the GAP containing $\hat{A}\cup \{0\}$ will be crucial for our application to non-averaging sets. 

To deduce Theorem \ref{thm:hom-AP} from Theorem \ref{thm:hom-AP-build}, one starts with the large homogeneous GAP $Q=csP$ in $\Sigma(A')$ guaranteed by Theorem \ref{thm:hom-AP-build} 
and adds elements of $A$ to enlarge the GAP. This is fairly straightforward in the one-dimensional case, because we can make use of the simple observation that, for any interval $I$ of length at least $n$ and any $0\le a\le n$, the sumset $I+\{0,a\}$ is an interval of length $|I|+a$. 
However, the multidimensional case is more subtle, as, given a large GAP $Q$, the set $Q+\{0,a\}$ is not necessarily a larger GAP of the same dimension. 
To circumvent this, instead of directly finding a large GAP inside $Q+\sum_{a\in A\setminus A'}\{0,a\}$, we first show that the set of points $Q+\sum_{a\in A\setminus A'}\{0,a\}$ is essentially the projection of the intersection of a convex body and a lattice and then use this structure and a discrete John-type theorem of Tao and Vu~\cite{TV} to find the desired large homogeneous GAP inside $\Sigma(A)$. 

The proof of our main technical result, Theorem \ref{thm:hom-AP-build}, has several steps, roughly as follows:
\begin{enumerate}
\item We preprocess $A$ to obtain a dense subset $\hat{A}$ of $A$ with certain useful properties. 
\item We partition $\hat{A}$ randomly into $\ell$ sets ${A}_1, \dots, {A}_\ell$ of roughly equal size which inherit these properties.
\item We show that there are subsets $A_i'$ of ${A}_i$, each of size $s/\ell$, such that $|\Sigma(A'_i)| \gg |\frac{s}{\ell} (\hat{A}\cup \{0\})|$, where $\frac{s}{\ell} (\hat{A}\cup \{0\})$ is the $\frac{s}{\ell}$-fold sumset of $\hat{A}\cup \{0\}$. 
\item We show that there is a GAP $P$ containing $\hat{A}\cup \{0\}$ such that $\frac{s}{\ell} (\hat{A}\cup \{0\})$ is dense in $\frac{s}{\ell} P$ and $c s P$ is proper. 
\item From the previous two steps, $\Sigma(A'_i)$ is dense in $\frac{s}{\ell} P$, allowing us to show that the sum of the sets $\Sigma(A'_i)$ contains a proper homogeneous translate of $c' sP$. 
\end{enumerate}  

A GAP $P$ containing $\hat{A}\cup \{0\}$ and satisfying the properties in Step 4 can be obtained by analyzing the growth of high-fold sumsets of $\hat{A}\cup \{0\}$, combined with an application of Freiman's celebrated theorem on the structure of sets with small doubling. In fact, we give an almost complete characterization of $h(\hat{A}\cup \{0\})$ for  large $h$ in terms of a GAP of appropriate dimension containing $\hat{A}\cup \{0\}$. Roughly speaking, for each sufficiently large $h$, there is a positive integer $d$, which we call the {\it $h$-dimension} of $\hat{A} \cup \{0\}$, and a GAP $P$ of dimension $d$ containing $\hat{A}\cup \{0\}$, which we call the {\it $d$-bounding box} of $\hat{A} \cup \{0\}$, such that $h(\hat{A}\cup \{0\})$ contains $chP$, where $chP$ is proper, for some appropriate constant $c>0$. Note that this is clearly optimal up to the constant $c$, as $h(\hat{A}\cup \{0\})$ is contained in $hP$.

The reason that we consider a subset $\hat{A}$ of $A$ rather than just $A$ itself is that there are examples where $\Sigma(A'_i)$ is not dense in $\frac{s}{\ell}(A \cup \{0\})$
for any subset $A'_i$ of $A$ of size $s/\ell$. For example, consider the case where $A$ consists of a progression $[m]$ and a single element much larger than $m$. We show that this is in some sense the only example: by performing a preprocessing step where we remove a small number of elements from $A$, we can guarantee that $\Sigma(A'_i)$ is dense in $\frac{s}{\ell}(\hat{A}\cup \{0\})$ and, hence, in $\frac{s}{\ell}P$. This preprocessing step replaces 
$A$ by a subset $\hat{A}$ of $A$ with the property that any reasonably large subset of $\hat{A}$ has similar behavior to $\hat{A}$ with respect to taking high-fold sumsets and, crucially, this property is inherited by the subsets in a random partition of $\hat{A}$. 
 
At this point, in order to find $A'_i$ such that $|\Sigma(A'_i)|$ is large, we use an iterative greedy process that grows $\Sigma(A'_i)$ one element at a time. Let $S_j$ be the set of elements in $A_i$ not yet picked and $\Sigma(j) = \Sigma(A_i\setminus S_j)$. 
To obtain bounds on the increment in $|\Sigma(j)|$ at each step, we observe a duality between the size of this increment and the multifold sumsets of $S_j$: roughly speaking, if all elements of $S_j$ are ``almost periods'' of $\Sigma(j)$, i.e., their addition does not increase $|\Sigma(j)|$ significantly, then there is a large $k$ for which $|kS_j|$ is small. Since $S_j$ is itself a reasonably large subset of $A_i$, the size of $kS_j$ is captured by an appropriate GAP containing $A_i$. This allows one to estimate $|\Sigma(j)|$ via the size of $h(\hat{A}\cup \{0\})$ for suitable $h$, ultimately leading to the desired claim in Step 3.  

While some of the steps in this strategy bear similarity to the method used in \cite{CFP} to handle the one-dimensional case, the strategy here is largely different and allows one to obtain a much more precise characterization of the structure of $A$ and its set of subset sums. 

\subsection{Non-averaging sets}

Recall that a subset $A$ of $[n]$ is {\it non-averaging} if no element of $A$ is the average of two or more other elements of $A$. The problem of estimating $h(n)$, the maximum size of a non-averaging subset of $[n]$, was first raised by Straus~\cite{Str}. However, his paper gives considerable credit to Erd\H{o}s, who had already asked the closely related problem of estimating the maximum size of a non-dividing subset of $[n]$, where a subset $A$ of $[n]$ is {\it non-dividing} if no element of $A$ divides the sum of two or more other elements of $A$. Because of this, the problem of estimating $h(n)$ is sometimes referred to as the Erd\H{o}s--Straus non-averaging sets problem.

In his original paper, Straus~\cite{Str} showed that $h(n) \geq e^{c \sqrt{\log n}}$ for some positive constant $c$, while, in a follow-up paper~\cite{EStr}, he and Erd\H{o}s showed that $h(n) = O(n^{2/3})$. The lower bound was improved to a polynomial by Abbott, who first showed~\cite{Abb1} that $h(n) = \Omega(n^{1/10})$ and then improved~\cite{Abb2} this bound to $h(n) = \Omega(n^{1/5})$. The current best lower bound, $h(n) = \Omega(n^{1/4})$, which we suspect to be tight, follows from a surprisingly simple construction due to Bosznay~\cite{Bosz}. Indeed, if we fix an integer $q$, then the set of integers consisting of $n_i = i q^3 + i(i+1)/2$ for $i = 1, 2, \dots, q-1$ is a non-averaging subset of $[n]$, where $n = q^4$.

The Erd\H{o}s--Straus upper bound of $h(n)  = O(n^{2/3})$ follows by exploiting a relationship between $h(n)$ and another function $H(n)$. Indeed, if we write $H(n)$ for the maximum integer for which there are two subsets of $[n]$ of size $H(n)$ whose sets of subset sums have no non-zero common element, then a result of Straus~\cite{Str} says that $h(n) \leq 2 H(n) + 2$. What Erd\H{o}s and Straus proved was that $H(n) = O(n^{2/3})$, which then implies the corresponding bound for $h(n)$. Similarly, using the Freiman--S\'ark\"ozy result on homogeneous progressions, Erd\H{o}s and S\'{a}rk\"{o}zy \cite{ES} were able to show that $H(n) = O(\sqrt{n \log n})$, which again yields a similar upper bound on $h(n)$. 

This method was pushed to its limit in our recent paper~\cite{CFP}, where we showed that $H(n) = O(\sqrt{n})$, which is best possible up to the constant factor, as may be seen by considering the sets $[1,c\sqrt{n}]$ and $[n-c\sqrt{n},n]$ for any $c<\sqrt{2}$. Thus, while we have $h(n) = O(\sqrt{n})$, it seems that new tools are needed to push the bound below $\sqrt{n}$. Our results on homogeneous GAPs are just such tools, allowing us to give the first significant improvement of the upper bound on $h(n)$ since Erd\H{o}s and S\'{a}rk\"{o}zy's 1990 paper.

\begin{thm}
\label{thm:non-avg}
There is a constant $C$ such that if $A$ is a subset of $[n]$ with the property that no element of $A$ is equal to the average of two or more other elements of $A$, then $|A| \le C n^{\sqrt{2}-1}(\log n)^{2}$. 
\end{thm}

The proof of Theorem \ref{thm:non-avg} makes use of Theorem \ref{thm:hom-AP-build}. As in the previous work on non-averaging sets, we first reduce Theorem \ref{thm:non-avg} to the problem of finding a long arithmetic progression in a certain set of subset sums. When $|A| < \sqrt{n}$, one generally does not expect to have such long arithmetic progressions. However, for $|A| > C n^{\sqrt{2}-1}(\log n)^{2}$, we can apply Theorem \ref{thm:hom-AP-build} to conclude that there is a large subset $\hat{A}$ of $A$ such that either $\Sigma(\hat{A})$ contains a long arithmetic progression or $\hat{A}$ is contained in a $2$-dimensional GAP $P$ where there is a large subset $A'$ of $\hat{A}$ such that $\Sigma(A')$ contains $(c|A'|)P$. 
In the former case, we are done. For the latter case, we can use the assumption that $A$, and hence $\hat{A}$, is non-averaging and a suitable induction hypothesis to show that $\hat{A}$ is not dense on any one-dimensional fiber of $P$. But this then gives an additional gain on the size of $\Sigma(A')$, which ultimately leads to a contradiction.

\subsection*{Notation}

For the sake of clarity of presentation, we omit floor and ceiling signs whenever they are not essential. We
also maintain the convention that all logarithms are base two unless otherwise specified. We use standard asymptotic notation throughout, though we will occasionally write $c_{P}$ and $C_{P}$ for constants depending on certain parameters $P$. 

\section{Multifold sumsets and GAPS}

In this section, we build towards the proof of Theorem \ref{thm:hom-AP-build} by proving a collection of disparate results, primarily about multifold sumsets and GAPs, some of which are of independent interest. We begin with a brief outline.

After first recalling some standard definitions in Subsection~\ref{subsec:prelim}, we prove, in Subsection~\ref{subsec:build-box}, a high-dimensional analogue of a result of Lev saying that the sumset of any sufficiently large collection of dense subsets of intervals, none of which is a subset of an arithmetic progression of common difference greater than one, must contain a long interval.

Subsection \ref{subsec:structural} then contains many of our main definitions and results. In particular, in Lemma \ref{lem:AP-from-doubling}, we show that for any $A\subseteq [0,n-1]$ with $0\in A$ and $h \ge n^{1/\beta}$, where $\beta \ge 1$ is a fixed constant, there is a GAP $P$ containing $A$ such that $P$ approximates $A$ with respect to taking $h$-fold sumsets, in that $hA$ is contained in $hP$ and contains a translate of $chP$ for some constant $c>0$ depending only on $\beta$. This approximation will play an important role in the proof of Theorem \ref{thm:hom-AP-build}. Informed by this result, we then introduce two key notions, the \textit{$h$-dimension} of $A$ and the \textit{$h$-bounding box} of $A$.

In Subsection \ref{subsec:non-proper}, we show that, given a not necessarily proper homogeneous GAP $A$ of large volume, either one can find a proper homogeneous GAP in $A$ with size at least a constant fraction of the volume of $A$ or a homogeneous GAP of smaller dimension with size at least a constant fraction of the size of $A$. Applied inductively, this then allows us to find large proper homogeneous GAPs inside non-proper homogeneous GAPs.

In our proof of Theorem \ref{thm:hom-AP-build},  
we will relate subset sums with multifold sumsets of certain large subsets of $A$. In order to control these multifold sumsets, we need that certain good properties hold not only for $A$, but also for all sufficiently large subsets of $A$. That is, the properties should be stable. Instead of defining the relevant properties directly, in terms of the multifold sumsets of $A$ and its subsets, which are hard to control, we define them indirectly through certain proxies for the structure of the multifold sumsets, namely, the notions of $h$-dimension and $h$-bounding box which were defined in Subsection \ref{subsec:structural}. These proxies are easier to handle and a simple iterative argument, described in Subsection \ref{subsec:processing}, shows that we can modify a set $A$ by removing a small number of elements so that the desired stability conditions, defined and studied in Subsection \ref{subsec:stability}, are satisfied.

Finally, in Subsection \ref{subsec:growing-sum}, we collect some simple results that will help control the growth in size of a set of subset sums as we add elements to the underlying set.

\subsection{Preliminaries}\label{subsec:prelim}

In this short subsection, we record a number of definitions which will be important throughout the paper. We first recall the definition of a generalized arithmetic progression and say what it means for such a progression to be proper.

\begin{defn}\label{def:proper-GAP}
A generalized arithmetic progression or \emph{GAP}, for short, is a set of the form $Q = \{x+\sum_{i=1}^{d}n_{i}q_{i}:0\le n_{i}\le w_{i}-1\}$,
where $d, x, q_1, \dots, q_d, w_1, \dots, w_d$ are integers with $d$ and all the $w_i$ positive. We refer to $d$ as the {\it dimension} of $Q$, $(w_1,\dots,w_d)$ as the {\it widths} of $Q$ and $(q_1,\dots,q_d)$ as the {\it differences} of $Q$. We also define the \textit{volume} of $Q$ by $\textrm{Vol}(Q)=\prod_{i=1}^{d}w_{i}$.
\end{defn}

\begin{defn}\
A GAP $\{x+\sum_{i=1}^{d}n_{i}q_{i}:0\le n_{i}\le w_{i}-1\}$ is said to be {\it $s$-proper} if, for all choices of $n_{i,j},n'_{i,j} \in [0,w_i-1]$ for $i \in [d]$ and $j \in [s]$, $\sum_{j=1}^{s}\sum_{i=1}^{d}n_{i,j}q_{i}=\sum_{j=1}^{s}\sum_{i=1}^{d}n'_{i,j}q_{i}$
if and only if $\sum_{j=1}^{s}n_{i,j}=\sum_{j=1}^{s}n'_{i,j}$ for
all $1\le i\le d$. In particular, we call a $1$-proper $d$-dimensional GAP a \emph{proper} $d$-dimensional GAP. 
\end{defn}

Recall that the $s$-fold sumset is defined by $sA = \{a_1+\dots+a_s : a_1,\dots,a_s\in A\}$. The following lemma is straightforward from the definition of properness. 

\begin{lem}
Let $Q$ be a $d$-dimensional GAP. If $sQ$ is a proper $d$-dimensional GAP, then $Q$ is an $s$-proper
$d$-dimensional GAP.
\end{lem}

Given a $d$-dimensional GAP $Q$ of the form $\{x+\sum_{i=1}^{d}n_{i}q_{i}:0\le n_{i}\le w_{i}-1\}$,
we can define a mapping $\phi_{Q}:Q\to\mathbb{Z}^{d}$ by choosing, for each element of $Q$, an arbitrary representation as $x+\sum_{i=1}^{d}n_iq_i$ with $n_i\in [0,w_i-1]$ for all $i\le d$ and setting $\phi_{Q}(x+\sum_{i=1}^{d}n_{i}q_{i})=(n_{1},\dots,n_{d})$. Going forward, we fix $\phi_Q$ for any given $Q$ and refer to it as the {\it identification map}.
Note that if $Q$ is proper, $\phi_Q$ gives a bijection between $Q$ and a box in $\mathbb{Z}^{d}$, while if
$Q$ is $s$-proper, $\phi_{Q}$ is a Freiman $s$-isomorphism (see, for example, \cite{TV} for the definition of a Freiman isomorphism). We will often write $\phi_Q^{-1}$ for the map $\phi_Q^{-1}:\mathbb{Z}^d \to \mathbb{Z}$ defined by $\phi_Q^{-1}(n_1,\dots,n_d)=x+\sum_{i=1}^{d}n_iq_i$, 
which is a one-sided inverse of $\phi_Q$. 

Observe that if $Q$ is a proper $d$-dimensional GAP with identification map $\phi_Q$, then $\phi_Q(Q)$ is a box in $\mathbb{Z}^d$. If this is the case, then, for any proper $d'$-dimensional GAP $P = \{a+\sum_{i=1}^{d'}n_ip_i\}$ which is a subset of $\phi_Q(Q)\subseteq \mathbb{Z}^d$, we have that $\phi_Q^{-1}(P)$ is a proper $d'$-dimensional GAP. Indeed, letting $p_i = (p_{i,1},\dots,p_{i,d})$, if 
$$\phi_{Q}^{-1}(a+\sum_{i=1}^{d'}x_ip_i)=\phi_Q^{-1}(a+\sum_{i=1}^{d'}y_ip_i),$$
then $$\sum_{j=1}^{d} q_j \sum_{i=1}^{d'}x_i p_{i,j} = \sum_{j=1}^{d} q_j \sum_{i=1}^{d'}y_i p_{i,j}.$$
Hence, by the properness of $Q$, 
$$ \sum_{i=1}^{d'}x_i p_{i,j} = \sum_{i=1}^{d'}y_i p_{i,j}$$
for all $j\le d$. Thus, $\sum_{i=1}^{d'}x_i p_i = \sum_{i=1}^{d'}y_ip_i$ and so it follows from the properness of $P$ that 
$x_i=y_i$ for all $i\in [d']$.

The definition of a homogeneous GAP below captures the idea that a GAP is homogeneous if, when appropriately extended, it passes through the origin.

\begin{defn}\label{def:vol-GAP}
A GAP $Q=\{x+\sum_{i=1}^{d}n_{i}q_{i}:0\le n_{i}\le w_{i}-1\}$ is \textit{homogeneous} if $\gcd(q_{1},\dots,q_{d})|x$. 
\end{defn}

In particular, note that a GAP $Q$ is homogeneous if and only if it can be written in the form $Q=\{\sum_{i=1}^{d}n_{i}q_{i}:a_i\le n_{i}\le b_i\}$. 
When $Q$ is homogeneous, we can use this observation to generalize the definition of multifold sumsets to non-integer values of $s$ as follows.

\begin{defn}
Let $Q$ be a homogeneous $d$-dimensional GAP given by $Q=\{\sum_{i=1}^{d}n_{i}q_{i}:a_{i}\le n_{i}\le b_{i}\}$ for some real numbers $a_1, \dots, a_d, b_1, \dots, b_d$ with $a_i < b_i$ for all $i = 1, 2, \dots, d$.
For a positive real number $c$, we then let $cQ=\{\sum_{i=1}^{d}n_{i}q_{i}:ca_{i}\le n_{i}\le cb_{i}\}$.
\end{defn}

Observe that for a homogeneous GAP $Q$, the definition of $cQ$ depends on the specific representation of $Q$ (that is, the choice of differences $q_1,\dots,q_d$ and intervals $[a_i,b_i]$). However, when $c$ is a positive integer, the $c$-fold sumset and this definition of $cQ$ agree. In particular, for a positive integer $c$, $cQ$ only depends on $Q$ as a set and not on the particular representation of $Q$.

We say that a GAP $Q$ is \emph{centered} if we can write $Q = \{\sum_{i=1}^{d}n_iq_i:a_i\le n_i\le b_i\}$ with $a_i \le 0 \le b_i$ for all $i\in [d]$. The following observation will be useful later.

\begin{claim}\label{claim:center}
If $Q$ is a GAP that contains $0$, then it is centered. 
\end{claim}

\begin{proof}
Let $Q = \{x+\sum_{i=1}^{d}n_iq_i:a'_i\le n_i\le b'_i\}$ be such that $0\in Q$. Then we can write $0 = x+\sum_{i=1}^{d}m_iq_i$, where $m_i \in [a'_i,b'_i]$. Hence, we can also write $Q = \{\sum_{i=1}^{d} n_iq_i:a'_i-m_i\le n_i\le b'_i-m_i\}$, where $a'_i-m_i \le 0 \le b'_i-m_i$ as $m_i \in [a'_i,b'_i]$. 
\end{proof}

We now record some further elementary results about GAPs for future use.

\begin{lem}\label{lem:GAP-size}
The following estimates hold:
\begin{enumerate}
    \item Let $s$ be a positive integer and $P$ a GAP of dimension $d$ where $sP$ is proper. Then $(s/2)^d|P|\le |sP| \le s^d|P|$. 
    \item Let $c>0$ and let $P$ be a homogeneous GAP of dimension $d$ whose minimum width is at least $2+2c^{-1}$. Then, if $cP$ is proper, $(c/2)^d|P|\le |cP|\le c^d|P|$.
\end{enumerate}
\end{lem}

\begin{proof}
Let $P = \{x + \sum_{i=1}^{d}n_iq_i:n_i \in [a_i,b_i]\}$ and let $w_i = b_i-a_i+1$. For the first bound, we use that $sP$ is a GAP of dimension $d$ with widths $s(w_1-1)+1,\dots,s(w_d-1)+1$ and note that $sw_i/2\le s(w_i-1)+1\le sw_i$. For the second bound, where $x = 0$, note that $cP$ is a GAP of dimension $d$ with widths $\lfloor cb_i \rfloor -\lceil ca_i\rceil +1 \ge c(b_i-a_i)-1 \ge (c/2)w_i$. 
\end{proof}

In the next two lemmas, we assume that the GAPs $A$ and $B$ are given with fixed representations and $cA$ and $cB$ are defined with respect to these representations. 

\begin{lem}\label{lem:cGAP-sub}
Let $c > 0$ and let $A = \{\sum_{i=1}^{d}n_iq_i:n_i\in [a_i,b_i]\}$ be a homogeneous GAP whose minimum width is at least $1 + 4 c^{-1}$. 
Then $2\lceil c^{-1}\rceil (cA)$ contains a translate of $A$. 
\end{lem}

\begin{proof}
Since $cA = \{\sum_{i=1}^{d}n_iq_i:n_i\in [ca_i,cb_i]\}$, we have that, for any positive integer $s$, $s (cA) =  \{\sum_{i=1}^{d}n_iq_i:n_i\in [s\lceil ca_i\rceil,s\lfloor cb_i\rfloor]\}$. Thus, to show that $2\lceil c^{-1}\rceil (cA)$ contains a translate of $A$, we only need to check that $2\lceil c^{-1}\rceil (\lfloor cb_i\rfloor - \lceil ca_i\rceil) \ge \lfloor b_i\rfloor - \lceil a_i\rceil$. However, this is true, since 
\[
2\lceil c^{-1}\rceil (\lfloor cb_i\rfloor - \lceil ca_i\rceil) \ge 2\lceil c^{-1}\rceil (c(b_i-a_i)-2) \ge 2c^{-1} \cdot \frac{1}{2}c(b_i-a_i) \ge \lfloor b_i\rfloor - \lceil a_i\rceil. \qedhere
\]
\end{proof}

\begin{lem}\label{lem:cGAP}
Let $0<c\le 1$ and let $A,B$ be homogeneous GAPs such that $cA$ is contained in a translate of $cB$ with the minimum width of $A$ 
at least $1+4c^{-1}$. Then $4B$ contains a translate of $A$. 
\end{lem}

\begin{proof}
By Lemma \ref{lem:cGAP-sub}, $2\lceil c^{-1}\rceil (cA)$ contains a translate of $A$. Thus, letting $cA+x$ be a translate of $cA$ contained in $B$, we have $2\lceil c^{-1}\rceil (cB)$ contains $2\lceil c^{-1}\rceil (cA+x)$ which contains a translate of $2\lceil c^{-1}\rceil (cA)$ and hence $A$. Furthermore, for any positive integer $s$, $s(cB)$ is contained in $(sc)B$ by definition. Hence, $A$ is contained in a translate of $(2\lceil c^{-1}\rceil c)B \subseteq 4B$. 
\end{proof}

\subsection{Building boxes from dense subsets}\label{subsec:build-box}

In this subsection, we generalize the following result of Lev~\cite{Lev} to higher dimensions.  

\begin{lem}[Lev~\cite{Lev}] \label{lem:Lev}
Suppose $\ell,q\ge1$ and $n\ge3$ are integers with
$\ell\ge2\lceil(q-1)/(n-2)\rceil$. If $S_{1},\dots,S_{\ell}$ are integer
sets each having at least $n$ elements, each a subset of an interval
of at most $q+1$ integers and none a subset of an arithmetic
progression of common difference greater than one, then $S_{1}+\cdots+S_{\ell}$
contains an interval of length at least $\ell(n-1)+1$. 
\end{lem}

The following lemma is a simple consequence of Lev's result. 
 
\begin{lem} \label{lem:build-interval}
Let $0<c<1$ and let $n,\ell,v$ be positive integers such that $n \geq 4/c$ and $\ell \ge 10/c$. Let
$A_{1},\dots,A_{\ell}$ be subsets of $[n]$ such that each $A_{i}$ satisfies $|A_{i}|\ge cn$ and is a subset of a translate of $v\mathbb{Z}$ but 
not a subset of any translate of a proper subgroup of $v\mathbb{Z}$.
Then $A_{1}+\dots+A_{\ell}$ contains a translate of $v\cdot [c\ell n/2]$.
\end{lem}

\begin{proof}
By assumption, there exist non-negative integers $a_i$ such that $S_i = \{(x-a_i)/v \,:\, x\in A_i\}$ is a subset of $\{0\} \cup [n/v]$ which is not contained in an arithmetic progression of common difference greater than one. 
By Lemma \ref{lem:Lev}, if $\ell \ge 2\lceil (n/v)/(cn-2)\rceil$, then $S_1+\dots+S_{\ell}$ contains an interval $I$ of length at least $\ell (cn-1)+1 > c\ell n/2$, where we used that $n \geq 4/c$. Since $v\ge 1$, the conditions $\ell \ge 10/c$ and $n \geq 4/c$ guarantee that $\ell \ge 2\lceil (n/v)/(cn-2)\rceil$. Thus, $A_{1}+\dots+A_{\ell}$ contains a translate of $v\cdot [c\ell n/2]$, as required.
\end{proof}

We will also need the following simple claim.

\begin{claim}\label{claim:abelian-expansion}
Let $G$ be a finite abelian group and let $A_1,\dots,A_{|G|}$ be subsets of $G$ such that no $A_i$ is contained in a translate of a proper subgroup of $G$. Then $A_1+\dots + A_{|G|}$ contains $G$.
\end{claim}

\begin{proof}
By translating each $A_i$, we may assume without loss of generality
 that $0\in A_i$ for all $i\le |G|$. We will show that if $A_i$ is not contained in a proper subgroup of $G$, then $|A_1+\dots +A_{i-1}+A_i| > |A_1+\dots +A_{i-1}|$ or $A_1+\dots +A_{i-1}=G$, from which the claim follows. Suppose, for the sake of contradiction, that for some $i\le |G|$ we have $|A_1+\dots +A_{i-1}+A_i| = |A_1+\dots +A_{i-1}|$. Let $S = A_1+\dots +A_{i-1}$. Then $|S+A_i|=|S|$. Since $0\in A_i$, we have that $S+A_i=S$. The set $P$ of elements $g\in G$ with $S+g=S$ is a subgroup of $G$ and if $S \ne \emptyset$ and $S \ne G$, then $P$ is proper. But $A_i \subseteq P$, contradicting our assumption that $A_i$ is not contained in a translate of a proper subgroup of $G$. 
\end{proof}

Before stating the main result of this subsection, we record some more definitions.

\begin{defn}
A box $Q$ in $\mathbb{Z}^d$ is a subset of $\mathbb{Z}^d$ of the form $Q = \{(x_1,\dots,x_d) \, : \,x_j \in I_j\}$, where the $I_j$ are non-empty intervals in $\mathbb{Z}$. We say that $(w_1,\dots,w_d)=(|I_1|,\dots,|I_d|)$ are the widths of $Q$.
\end{defn}

\begin{defn}
A subset $A$ of $\mathbb{Z}^d$ is said to be {\it reduced} if, for any proper subgroup $H$ of $\mathbb{Z}^d$ of the form $v_1\mathbb{Z} \times \dots \times v_d\mathbb{Z}$, $A \bmod H$ is not contained in a translate of a proper subgroup of $\mathbb{Z}^d/H$. Similarly, we say that a subset $A$ of a $d$-dimensional GAP $Q$ is {\it reduced} if, under the identification map $\phi_Q:Q\to \mathbb{Z}^d$, $A$ is mapped to a reduced subset of $\mathbb{Z}^{d}$.
\end{defn}

Our higher-dimensional generalization of Lev's result is now as follows.

\begin{lem} \label{lem:ell-box} 
For any $0<c \le 1$ and positive integer $d$, there exists a constant $\gamma > 0$ such that the following holds. Let $A_{1},\dots,A_{\ell}$ be reduced subsets of a box $Q$ in $\mathbb{Z}^{d}$ such that $|A_{i}|\ge c|Q|$ for $1 \leq i \leq d$. Then, assuming $\ell$ and the minimum width of $Q$ are sufficiently large in terms of $c$ and $d$, $A_{1}+\dots+A_{\ell}$ contains a translate of $\gamma \ell Q$. 
\end{lem}

\begin{proof}
For $x = (x_1, \dots, x_d) \in\mathbb{Z}^{d}$,
let $\pi_{1}(x)=x_{1}$ be the projection onto the first coordinate
of $x$ and $\pi_{1}'(x)=(x_{2},\dots,x_{d})$ be the projection onto
the remaining coordinates. For any $A \subseteq \mathbb{Z}^d$ and $y\in\mathbb{Z}^{d-1}$, 
let $\pi_{1}(A,y)=\{u:(u,y)\in A\}$.

Let the widths of $Q$ be $(w_{1},\dots,w_{d})$. Without loss of generality, by translation, we can assume that $Q$ contains $0$.
Since each set $A_{i}$ with $1 \leq i \leq d$ has density at least $c$ in $Q$, for at
least a $c/2$-fraction of the elements $y \in \pi_{1}'(Q)$, 
we have $|\pi_{1}(A_{i},y)|/w_{1}\ge c/2$.
We define $\tilde{A}_{i}=\{(u,y)\in A_{i}:|\pi_{1}(A_{i},y)|/w_{i}\ge c/2\}$.
Then $|\tilde{A}_{i}|\ge c^2|Q|/4$ 
and $|\pi_{1}'(\tilde{A}_{i})|\ge c|\pi_{1}'(Q)|/2$.
Furthermore, for each $y\in\mathbb{Z}^{d-1}$ such that $\pi_{1}(\tilde{A}_{i},y)$
is non-empty, we have $|\pi_{1}(\tilde{A}_{i},y)|\ge cw_{1}/2$. 

Let $\alpha = 1/(2d)$. For each $1 \leq i \leq \alpha \ell$, choose $y_i$ such that $B_i := \pi_{1}(\tilde{A}_{i},y_i)$
is non-empty, in which case it is a subset of $[w_1]$ of density at least $c/2$. Then, for $z_{1} = \sum_{i=1}^{\alpha \ell}y_i$, the sumset  $\sum_{i=1}^{\alpha\ell}\tilde{A}_{i}$ contains all elements of the form $(u,z_{1})$ with $u\in B_{1}+\dots+B_{\alpha\ell}$. 
Observe that if $B_{i}$ is contained in a translate of a subgroup
$v_i \mathbb{Z}$ of $\mathbb{Z}$, then $v_{i} \le 4/c$. 
Therefore, by the pigeonhole principle, we can find some $v \le 4/c$ and at least $c\alpha\ell/4$ of the sets $B_{i}$ that are
contained in a translate of $v\mathbb{Z}$ but not in a translate
of any proper subgroup of $v\mathbb{Z}$.
By Lemma~\ref{lem:build-interval}, the sum of these sets contains
a translate of $v\cdot [c^2 \alpha\ell w_1/16]$. Thus, $\sum_{i\leq\alpha\ell}A_{i}$ contains all elements of the form $(u,z_{1})$
where $u$ is in a translate of $v_{1}\cdot[c^2 \alpha \ell w_1 /16]$
for some $v_{1}\le 4/c$. 

By a similar argument, defining $\pi_{k}(x)=x_{k}$ and $\pi'_{k}(x)=(x_{1},\dots,x_{k-1},x_{k+1},\dots,x_{d})$,
we obtain that for some $z_{k}\in\mathbb{Z}^{d-1}$, $\sum_{\alpha(k-1)\ell < i \leq \alpha k\ell}A_{i}$
contains all elements $x$ such that $\pi'_{k}(x)=z_{k}$ and $\pi_{k}(x)$
is in a translate of $v_{k}\cdot[c^2 \alpha \ell w_k/16]$
for some $v_{k}\le 4/c$. 

Hence, $\sum_{1\le i\le d\alpha\ell}A_{i}$ contains a translate of the
dilated box $\{(b_{1}v_{1},\dots,b_{d}v_{d}):b_{k}\in[0,c'\alpha\ell w_{k}]\}$
for some $c'$ depending only on $c$. Consider the
subgroup $H=v_1\mathbb{Z}\times\dots\times v_d\mathbb{Z}$ of $\mathbb{Z}^d$ and let $G = \mathbb{Z}_{v_1}\times \dots \times \mathbb{Z}_{v_d}$. Since
each set $A_{i}$ is reduced, $A_i \bmod H$ is not contained in a translate of any proper subgroup of $G$. Thus, by Claim \ref{claim:abelian-expansion}, the set $A_{\alpha d\ell+1}+\dots+A_{\alpha d\ell+v_{1}v_{2}\cdots v_{d}}$ modulo $v_1\mathbb{Z}\times\dots\times v_d\mathbb{Z}$ is equal to $\mathbb{Z}_{v_{1}}\times\dots\times\mathbb{Z}_{v_{d}}$.

Assume now that $\ell$ is sufficiently large in terms of $c,c'$ and $d$. We claim that $\sum_{1\le i\le\alpha d\ell + v_{1}v_{2}\cdots v_{d}}A_{i}$ contains a translate of the box with widths
$(c''\alpha\ell w_{1},\dots,c''\alpha\ell w_{d})$
for $c''=c'/4$,
which immediately yields the desired conclusion of the lemma. To verify this claim, suppose that $\sum_{1\le i\le d\alpha\ell}A_{i}$ contains $t+\{(b_{1}v_{1},\dots,b_{d}v_{d}):b_{k}\in[0,c'\alpha\ell w_{k}]\}$. We next show that each element in a particular box is in the sumset. Consider an element of the form $x = t+(\lfloor c'\alpha \ell w_1 / 2\rfloor v_1 + z_1,\dots,\lfloor c'\alpha \ell w_d / 2\rfloor v_d + z_d)$ with $|z_k|\le c''\alpha\ell w_k$ for all $k \in [d]$. Since $A_{\alpha d\ell+1}+\dots+A_{\alpha d\ell+v_{1}v_{2}\cdots v_{d}}$ modulo $v_1\mathbb{Z}\times\dots\times v_d\mathbb{Z}$ is equal to $\mathbb{Z}_{v_{1}}\times\dots\times\mathbb{Z}_{v_{d}}$, we can find an element $r \in A_{\alpha d\ell+1}+\dots+A_{\alpha d\ell+v_{1}v_{2}\cdots v_{d}}$ such that $r-z \in v_1\mathbb{Z}\times\dots\times v_d\mathbb{Z}$ for $z=(z_1,\dots,z_d)$. Let $r - z = (a_1 v_1,\dots,a_d v_d)$. We have 
$|a_k| \le c''\alpha\ell w_k + v_{1}v_{2}\cdots v_{d} w_k \le (c''\alpha \ell + (4/c)^d)w_k$. Thus, $\lfloor c'\alpha \ell w_k / 2\rfloor-a_k \in [0, c'\alpha \ell w_k]$, using the assumption that $c''=c'/4$ and that $\ell$ is sufficiently large in $c,c'$ and $d$. Hence, $t + ((\lfloor c'\alpha \ell w_1 / 2\rfloor-a_1) v_1 , \dots , (\lfloor c'\alpha \ell w_d / 2\rfloor-a_d) v_d) \in \sum_{1\le i\le d\alpha\ell}A_{i}$. Therefore, 
$$x = r + t + ((\lfloor c'\alpha \ell w_1 / 2\rfloor-a_1) v_1 , \dots , (\lfloor c'\alpha \ell w_d / 2\rfloor-a_d) v_d) \in \sum_{1\le i\le\alpha d\ell+ v_{1}v_{2}\cdots v_{d}}A_{i}.$$
In particular, we have that $\sum_{1\le i\le\alpha d\ell+ v_{1}v_{2}\cdots v_{d}}A_{i}$ contains the box with widths $(c''\alpha\ell w_{1},\dots,c''\alpha\ell w_{d})$ centered at $t+(\lfloor c'\alpha \ell w_1 / 2\rfloor v_1,\dots,\lfloor c'\alpha \ell w_d / 2\rfloor v_d)$. 
\end{proof}

Our next lemma is a technical generalization of Lemma \ref{lem:ell-box} to the case where $A$ is not necessarily reduced. In the statement and proof, given a subset $A$ of a group $G$, we write $\langle A\rangle$ for the group generated by $A$, which is a subgroup of $G$. We also emphasize that a GAP here is a subset of $\mathbb{Z}^d$ rather than of $\mathbb{Z}$.

\begin{lem}\label{lem:non-reduced}
For $0<c \leq 1$ and a box $Q$ in $\mathbb{Z}^{d}$, let $A$ be a subset of $Q$ with $|A|\ge c|Q|$ and $0\in A$ and let $\Gamma=\langle A\rangle$. Then, assuming the minimum width of $Q$ is sufficiently large in terms of $c$ and $d$, there exists a positive constant $\kappa$ depending only on $c$ and $d$ and a GAP $P$ in $\mathbb{Z}^d$ with differences $p_1,\dots,p_d \in \Gamma$ forming a basis for $\Gamma$
such that $P$ is contained in a translate of $\kappa Q$ and $P$ contains $Q\cap \Gamma$. Furthermore, for $\ell$ sufficiently large in terms of $c$ and $d$, the multifold sumset $\ell A$ contains a translate of $\gamma \ell P$ for some constant $\gamma>0$ depending only on $c$ and $d$. 
\end{lem}

\begin{proof}
Let $(w_1,\dots,w_d)$ be the widths of $Q$. From the proof of Lemma \ref{lem:ell-box}, for $\ell_0$ sufficiently large in terms of $c$ and $d$, $\ell_0 A$ contains a translate of a  dilated box $\tilde{Q} = \{(b_{1}v_{1},\dots,b_{d}v_{d}):b_{k}\in[0, w_{k}]\}$, where $v_1,\dots,v_d$ are bounded in terms of $c$ and $d$. 

Since $\ell_0 A$ contains a dilated box of dimension $d$, we have that $\langle A\rangle$ has dimension $d$ and, hence, the subgroup $\Gamma = \langle A\rangle$ of $\mathbb{Z}^{d}$ has a basis $(p_1,\dots,p_d)$. 
Note that $\Gamma$ contains $\ell_0 A$, which in turn contains a translate of $\tilde{Q}$. Hence, $\Gamma$ contains $\tilde{Q}$. 

We claim that we can choose a basis for $\Gamma$ such that the basis elements lie in $\prod_{i=1}^{d}[0,2v_i-1]$. Indeed, for any basis $(p_1,\dots,p_d)$, any $i,j\in [d]$ and any integer $r$, if we write $e_i$ for the standard basis vector which is $1$ in the $i$th coordinate and $0$ otherwise, then either $(p_1,\dots,p_{j-1},p_j-rv_ie_i,p_{j+1},\dots,p_d)$ or $(p_1,\dots,p_{j-1},p_j-(r+1)v_ie_i,p_{j+1},\dots,p_d)$ form a basis of $\Gamma$. By iterating this step, we can form a basis such that the $i$th coordinate of each vector in the basis is in $[0,2v_i-1]$. Thus, we can assume that the basis $(p_1,\dots,p_d)$ consists of elements in $\prod_{i=1}^{d}[0,2v_i-1]$. 

Let $P_0$ be the GAP $\{\sum_{i=1}^{d} n_i p_i : -w_i \le n_i \le w_i\}$. 
Then, for some constant $\xi$ depending only on $v_1,\dots,v_d$, $\xi P_0$ contains $Q \cap \Gamma$. Indeed, since $p_1,\dots,p_d$ form a basis of $\Gamma$, each element of $\prod_{i=1}^{d}[0,v_i-1] \cap \Gamma$ and each of $v_1e_1,\dots,v_de_d$ can be written as a linear combination of $p_1,\dots, p_d$. Let $\Xi$ be the maximum absolute value of a coefficient appearing in any of these linear combinations, noting that $\Xi$ is bounded in terms of $c$ and $d$. 
For each element $y=(y_1,\dots,y_d)$ of $\Gamma$, we can write $y=z+t$ where $z_i = v_i \lfloor y_i/v_i\rfloor$ and $t_i = y_i - v_i \lfloor y_i /v_i \rfloor$. Then $z$ is a linear combination of the $v_ie_i$ which is contained in $\Gamma$ and so $t$ is an element of $\prod_{i=1}^{d}[0,v_i-1] \cap \Gamma$. Furthermore, recalling that $0\in A\subseteq Q$, if $y\in Q\cap \Gamma$, then $z$ is a linear combination of the $v_ie_i$, each with coefficient at most $w_i$ in absolute value. Thus, any element of $Q\cap \Gamma$ can be written as a linear combination of $p_1,\dots,p_d$ with the absolute value of the $i$th coefficient bounded by $(w_i+1)\Xi \le 2w_i \Xi$. Thus, the claim holds with $\xi = 2\Xi$. 

Let $P=\xi P_0$. Since $A\subseteq Q\cap \Gamma$, we have $A\subseteq P$. Furthermore, since the $j$th coordinate of $p_i$ is bounded by $2v_j-1$, which is bounded in terms of $c$ and $d$, and since $\xi$ is bounded in terms of $c$ and $d$, $P$ is contained in a translate of $\kappa Q$ for some constant $\kappa$ depending only on $c$ and $d$. 

We have that $H=v_1 \mathbb{Z}\times \dots \times v_d\mathbb{Z}$ is a subgroup of $\Gamma$. Let $G=\Gamma / H$, so $|G|\le v_1\cdots v_d$. Note that $A$ is not contained in any proper subgroup of $G$ by the definition of $\Gamma$ and, since $0\in A$, $A$ is not contained in a translate of any proper subgroup of $G$. By Claim \ref{claim:abelian-expansion}, $v_1v_2\cdots v_d A \bmod v_1\mathbb{Z}\times\dots \times v_d \mathbb{Z}$ contains $\Gamma \bmod v_1\mathbb{Z}\times\dots \times v_d \mathbb{Z}$. Thus, as in the proof of Lemma~\ref{lem:ell-box}, $\ell_0A + v_1v_2\cdots v_d A$ contains a translate of $\gamma_0 \ell_0 P$ for some constant $\gamma_0$ depending only on $c$ and $d$. Hence, for $\ell=\ell_0+v_1v_2\cdots v_d$, we have that $\ell A$ contains a translate of $\gamma\ell P$, where $\gamma = \gamma_0/(1+v_1v_2\cdots v_d/\ell_0)$. 
\end{proof}

For future use, we record the following corollary of Lemma \ref{lem:non-reduced}. Here the {\it greatest common divisor} $\gcd(P)$ of a GAP $P$ is the greatest common divisor of the differences of $P$ and, for a general subset $A$ of $\mathbb{Z}$, $\gcd(A)$ is the greatest common divisor of the elements of $A-A$. Note that when $A$ is a GAP these two notions coincide. For $A\subseteq \mathbb{Z}^d$, its affine span $\overline{\langle A\rangle}$ is $a+\langle A-A\rangle$ for some $a\in A$, noting that this definition does not depend on the choice of $a$. Furthermore, if $0\in A$, then the affine span $\overline{\langle A\rangle}$ and the span $\langle A\rangle$ coincide. 

\begin{cor}\label{cor:non-reduced}
Let $0<c \le 1$. The following claims hold: 
\begin{enumerate}
    \item Let $Q$ be a box of dimension $d$ in $\mathbb{Z}^{d}$ and let $A$ be a subset of $Q$ with $|A|\ge c|Q|$. Assume that the minimum width of $Q$ is sufficiently large in terms of $c$ and $d$. Then there is a positive constant $\kappa$ depending only on $c$ and $d$ and a $d$-dimensional GAP $P$ of dimension $d$ in $\mathbb{Z}^d$ such that $P$ is contained in a translate of $\kappa Q$, $P$ contains $Q\cap \overline{\langle A\rangle}$ 
    and $\kappa A$ contains a translate of $P$. Furthermore, if $0\in A$, then one may assume that $P$ is centered. 
    \item Let $Q$ be a proper GAP of dimension $d$ and let $A$ be a subset of $Q$ with $|A|\ge c|Q|$. Assume that the minimum width of $Q$ is sufficiently large in terms of $c$ and $d$. Then there is a positive integer $\kappa$ depending only on $c$ and $d$ and a proper GAP $P$ of dimension $d$ such that $|P|\le \kappa |Q|$, $\kappa P$ contains a translate of $A$, $\kappa A$ contains a translate of $P$ and $\gcd(P)=\gcd(A)$. 
\end{enumerate}
\end{cor}

\begin{proof}
The first claim follows directly from Lemma \ref{lem:non-reduced} if $0\in A$. If $0\notin A$, consider the translate $A-a$ for some element $a\in A$, so $0\in A-a$. Then the claim holds for the box $Q-a$ and the subset $A-a$ of $Q-a$, so there is a GAP $P$ of dimension $d$ in $\mathbb{Z}^d$ such that $P$ is contained in a translate of $\kappa (Q-a)$, $P$ contains a translate of $(Q-a)\cap \langle A-a \rangle$ and $\kappa (A-a)$ contains a translate of $P$. Then $P+a$ is contained in a translate of $\kappa Q$, $P+a$ contains $Q\cap \overline{\langle A\rangle}$ and $\kappa A$ contains a translate of $P+a$. The GAP $P+a$ therefore has the required properties. 

For the second claim, observe that we can assume without loss of generality that $0\in A$ as the hypotheses and conclusions are invariant under translation. By the first claim, under the identification map $\phi_Q:Q\to \mathbb{Z}^d$, one can find a positive constant $\kappa$, which we can assume to be an integer, and a GAP $\tilde{P}$ in $\mathbb{Z}^d$ such that $\tilde{P}$ is contained in a translate of $\kappa \phi_Q(Q)$, $\tilde{P}$ contains $\phi_Q(A)$ and $\kappa \phi_Q(A)$ contains a translate of $\tilde{P}$. Note that for a box $Y$ in $\mathbb{Z}^d$ and $X \subset \mathbb{Z}^d$, if $\kappa X \subseteq \kappa Y$, then $X \subseteq Y$. 
Hence, $\kappa^{-1}\tilde{P}$ is contained in a translate of $\phi_Q(Q)$. Since $Q$ is proper and $\kappa^{-1}\tilde{P}$ is proper, 
$\phi_Q^{-1}(\kappa^{-1}\tilde{P})$ is also proper. Let $P :=\phi_Q^{-1}(\kappa^{-1}\tilde{P})$. Note that $\phi_Q^{-1}:\mathbb{Z}^{d}\to \mathbb{Z}$ is a linear map, so that, since $4\kappa$ is a positive integer, any element in $\phi_Q^{-1}(4\kappa \cdot \kappa^{-1}\tilde{P})$ is contained in $4\kappa \phi_Q^{-1}(\kappa^{-1}\tilde{P})$. Hence, by Lemma \ref{lem:cGAP-sub}, 
$4\kappa P \supseteq \phi_Q^{-1}(4\kappa \cdot \kappa^{-1}\tilde{P})$ contains a translate of $\phi_Q^{-1}(\tilde{P})\supseteq A$. 
Furthermore, $\kappa A$ contains a translate of $\phi_Q^{-1}(\tilde{P})$ and, hence, $P$. Finally, $|P|\le |\tilde{P}|\le |\kappa \phi_Q(Q)|\le \kappa^{d}|Q|$  and $\gcd(P)=\gcd(A)$, since $\kappa A$ contains a translate of $P$ and $\kappa P$ contains a translate of $A$. Hence, upon replacing $\kappa$ with a larger constant, the GAP $P$ has the required properties. 
\end{proof}

\subsection{Structural results}\label{subsec:structural}

In this subsection, we give an approximation result for $h$-fold sumsets of sets $A\subseteq [0,n-1]$ with $0\in A$ and $h\ge n^{1/\beta}$ for some $\beta>1$, in the sense that we find a GAP $P$ with $A\subseteq P$ such that $hA$ contains a proper translate of $chP$ for some constant $c>0$ depending only on $\beta$. That is, $hA$ contains a translate of $chP$, while $hP$ contains $hA$, so we obtain ``upper and lower bounds'' on $hA$ that are tight up to the constant $c$. 

Given a GAP $P = \{x+\sum_{i=1}^{d}n_{i}q_{i}:0\le n_{i}\le w_{i}-1\}$ and a positive integer $f$, we refer to the set $\{x+\sum_{i=1}^{\min(d,f)} n_{i}q_{i}:0\le n_{i}\le w_{i}-1\}$ as the {\it first $f$ dimensions} of $P$. The next result follows from combining the main results of Bilu \cite{Bilu}, which build on the seminal work of Freiman \cite{Frei1,Frei2} on the structure of subsets $A$ of $\mathbb{Z}$ for which  $|2A|/|A|$ is small. 

\begin{lem}[Theorems 1.2 and 1.3 of \cite{Bilu}] \label{lem:bilu}
For any positive integers $s$ and $d$ and any $\delta > 0$, there exists 
$C$ such that the following holds. For any finite set of integers $A$ with  $|2A|\le2^{d+1-\delta}|A|$,
there exists an $s$-proper GAP $\tilde{Q}$ such that $2A$ is contained in $\tilde{Q}$, $|\tilde{Q}|\le C |2A|$ and if $Q$ is the GAP with dimension at most $d$
corresponding to the first $d$ dimensions of $\tilde{Q}$, then $|Q|\ge C^{-1}|\tilde{Q}|$.
\end{lem}

Observe that if $Q$ is a $d$-dimensional GAP with widths $(w_1,\ldots,w_d)$, then $|2Q|\leq \prod_{i=1}^{d}(2w_i-1)$, with equality if and only if $2Q$ is proper. The following lemma improves on this upper bound on $|2Q|$ by an additive factor of $\textrm{Vol}(Q)$ if $Q$ is not proper. It is best possible and gives a slight quantitative improvement on a result of Szemer\'edi and Vu~\cite[Lemma 4.2]{SV2}. 

\begin{lem} \label{lem:nonproper-doubling} If $Q$ is a $d$-dimensional GAP with widths $(w_1,\ldots,w_d)$ which is not proper, then 
$|2Q|+\textrm{Vol}(Q) \leq \prod_{i=1}^{d}(2w_i-1)$. Hence, $|2Q| < (2^{d}-1)\textrm{Vol}(Q) < 2^{d-c_d}\textrm{Vol}(Q)$, where $c_d=2^{-d}$. 
\end{lem}

\begin{proof} 
The $d$-dimensional GAP $Q$ has the form $Q=\{x+\sum_{i=1}^{d}n_{i}q_{i}:0\le n_{i}\le w_{i}-1\}$ and satisfies $\textrm{Vol}(Q)=\prod_{i=1}^d w_i$. 
Consider the linear map $\psi:\mathbb{Z}^d \rightarrow \mathbb{Z}$ given by $\psi(n_1,\ldots,n_d)=x+\sum_{i=1}^d n_iq_i$. If $Q$ is not proper, there 
are distinct $v=(n_1,\ldots,n_d)$, $v'=(n_1',\ldots,n_d')$ with $0 \leq n_i,n_i' \leq w_i-1$ for $i \in [d]$ which satisfy $\psi(v)=\psi(v')$. Let $z:=v-v' \in \mathbb{Z}^d \setminus \{0\}$, so that $\psi(z)=\psi(v)-\psi(v')=0$. Furthermore, letting $z=(z_1,\ldots,z_d)$, we have $|z_i| \leq w_i-1$ for $i \in [d]$. 

Consider the box $B=[0,2w_1-2] \times \cdots \times [0,2w_d-2] \subset \mathbb{Z}^d$. Call $b \in B$ \emph{compressed} if $b-z \not \in B$. Note that $|2Q|$ is at most the number of 
compressed elements in $B$ as $2Q=\{\psi(b):b~\textrm{compressed}\}$.  Observe that if $b=(b_1,\ldots,b_d) \in B$ satisfies, for all $i \in [d]$, that $b_i \geq w_i-1$ if $z_i \geq 0$ and $b_i \leq w_i-1$ if $z_i<0$, then $b$ is not compressed. So there are at least $\prod_{i=1}^d w_i=\textrm{Vol}(Q)$ elements in $B$ which are not compressed. Hence, $|2Q| \leq |B|-\textrm{Vol}(Q) =\prod_{i=1}^d (2w_i-1) - \textrm{Vol}(Q)$. 
\end{proof} 

The following corollary of Lemma \ref{lem:nonproper-doubling} will also be crucial in the next subsection when we come to study non-proper GAPs. 

\begin{lem}\label{lem:non-proper-d-1}
For every positive integer $d$, there exist $c_d'>0$ and a positive integer $C_d'$ such that the following holds. Let $A$ be a proper homogeneous $d$-dimensional GAP with all widths at least $C_d'$. If $2A$ is not proper, then $C_d'A$ contains a proper homogeneous $(d-1)$-dimensional GAP $Q$ of size at least $c_d' |A|$. Furthermore, $\gcd(Q)=\gcd(A)$ and $C_d' Q$ contains a translate of $A$. 
\end{lem}

\begin{proof}
First, we claim that, without loss of generality, we may assume that $0\in A$. Indeed, let $a\in A$, so that $0\in A-a$. Then $2A$ is proper if and only if $2(A-a)$ is proper. Assume that $C_d'(A-a)$ contains a proper homogeneous $(d-1)$-dimensional GAP $Q'$ of size at least $c_d'|A|$ such that $\gcd(Q')=\gcd(A-a)=\gcd(A)$ and $C_d'Q'$ contains a translate of $A-a$. Then, since $\gcd(Q')=\gcd(A)$ and $A$ is homogeneous, we have $\gcd(Q')|a$, so $C_d'A$ contains the proper homogeneous $(d-1)$-dimensional GAP $Q=Q'+C_d'a$ of size at least $c_d'|A|$, where $\gcd(Q)=\gcd(A)$ and $C_d'Q$ contains a translate of $A$. 

Suppose then that $0 \in A$. 
Since $2A$ is not proper, Lemma \ref{lem:nonproper-doubling} implies that 
\[
|4A| \le 2^{d-c_{d}}\textrm{Vol}(2A).
\]
We have that either $
|2A|\le 2^{d-c_d/2}|A|$
or 
$
|2A| > 2^{d-c_d/2}|A|
$, in which case, since $A$ is proper, $|2A| >2^{d-c_d/2}\textrm{Vol}(A) > 2^{-c_d/2}\textrm{Vol}(2A)$,
so, by the inequality above, $|4A| \le 2^{d-c_d/2}|2A|$. Hence, for either $r=1$ or $r=2$, we have $|2^rA|\le 2^{d-c_d/2}|2^{r-1}A|$. 

By Lemma \ref{lem:bilu}, for either $r=1$ or $r=2$, $2^rA$ is contained in a $2$-proper GAP $\tilde{Q}$ such that $|\tilde{Q}|\le K_{d-1}'|2^rA| \le K_{d-1}|2A|$ and the first $(d-1)$-dimensions $Q'$ of $\tilde{Q}$ satisfies $|Q'|\ge K_{d-1}^{-1}|\tilde{Q}|$, where $K_{d-1}$ and $K_{d-1}'$ are constants depending only on $d$. In particular, as $A\subseteq 2A$, this property holds for $r=1$. 
Let $\tilde{Q}=Q'\oplus W$. Since $0\in A$, we have $0\in \tilde{Q}$ and we can assume, without loss of generality, that $0\in Q'$ and $0\in W$. 

We claim that for some positive integer $C'$ depending only on $K_{d-1}$, $C'^{-1}A$ is contained in $C'^{-1}Q'$. Indeed, assume that there exists $x\in C'^{-1}A$ which is not contained in $C'^{-1}Q'$. Note that, since $0\in A$, $cA\subseteq A$ for all $0<c\le 1$. Thus, we have that $tx \in \tilde{Q}$ for all positive integers $t\le C'$. 
Since $\tilde{Q}$ is $2$-proper, each $tx$ has a unique representation as $w_t+q_t$, where $w_t\in W$ and $q_t \in Q'$. Therefore, since $2tx = (t-1)x + (t+1)x$ and again using that $\tilde{Q}$ is $2$-proper, we have $2w_t=w_{t-1}+w_{t+1}$ and $2q_t=q_{t-1}+q_{t+1}$ for all $t\in [1,C')$, where $w_0=q_0=0$. In particular, $q_t=tq_1$ and $w_t = tw_1$. Thus, if $w_1\ne 0$, then $w_1,\dots,w_{C'}$ are distinct, which implies that $|W| = |\tilde{Q}|/|Q'| \ge C'$, a contradiction for sufficiently large $C'$. On the other hand, if $w_1=0$, then $tx = q_t  \in Q'$ for all $t\in [C']$, which implies that $x \in C'^{-1}Q'$. 

From the above claim together with Lemma \ref{lem:cGAP}, we have that $4Q'$ contains a translate of $A$. 
Furthermore, $|4Q'|\ll_d |\tilde{Q}| \ll_d |A|$.  
Thus, by applying the second claim of Corollary \ref{cor:non-reduced} to $4Q'$, we can find a constant $C_d'$ and a proper GAP $P$ of dimension at most $d-1$ such that  $C_d'A$ contains a translate of $P$, 
$C_d'P$ contains a translate of $A$ and $\gcd(P)=\gcd(A)$. Since $\gcd(P)=\gcd(A)$ and $0\in A$, the translate $Q$ of $P$ contained in $C_d'A$ is homogeneous. Finally, since $C_d'P$ contains a translate of $A$, so does $C_d'Q$ and, hence, $|Q|\ge c_d'|A|$ for some $c_d' > 0$.  
\end{proof}

From Lemma \ref{lem:non-proper-d-1}, we obtain the following corollary. 
\begin{cor}\label{cor:nonproper-doub}
For every $d$, there exists $C_d>0$ such that if $A$ is a proper homogeneous $d$-dimensional GAP and $2A$ is not proper, then, for all $k \ge 1$, $|kA|\le k^{d-1} C_d|A|$. 
\end{cor}

\begin{proof}
We prove the statement by induction on $d$. The base case $d=1$ is trivial. 

Let $d\ge 2$. We first assume that all the widths of $A$ are at least $C_d'$, where $C_d'$ is the constant in Lemma \ref{lem:non-proper-d-1}. By Lemma \ref{lem:non-proper-d-1}, we can find a $(d-1)$-dimensional GAP $Q$ such that $C_d'Q$ contains a translate of $A$ and $C_d'A$ contains a translate of $Q$. Hence, 
\[
|kA| \le |kC_d'Q| \le (kC_d')^{d-1}|Q| \le (kC_d')^{d-1}|C_d'A| \le (kC_d')^{d-1} \cdot C_d'^d|A| \le k^{d-1}C_d|A|, 
\]
assuming that $C_d$ is chosen so that $C_d\ge C_d'^{2d-1}$. 

Next, consider the case where $A = \{\sum_{j=1}^{d} x_jq_j:x_j\in I_j\}$ with $I_j$ an interval of length $w_j$, $w_1 \le w_2 \le \dots \le w_d$ and $w_1 < C_d'$. Let $A'=\{\sum_{j=2}^{d} x_jq_j:x_j\in I_j\}$. If $2A'$ is not proper, then, by the induction hypothesis, $|kA'|\le k^{d-2}C_{d-1}|A'|$. Hence, \[ |kA| \le |kI_1 + kA'| \le C_d'C_{d-1}k^{d-1}|A|\le k^{d-1}C_d|A|,\]
assuming that $C_d \ge C_d'C_{d-1}$. 

Thus, we may assume that $2A'$ is proper. To handle this case, note that, since $2A$ is not proper, $cq_1 \in 2(A'-A')$ for some integer $c\in (0,2C_d']$. Observe that $I_1$ is contained in a translate of $\{1,2,\dots,c-1\} + cJ_1$ for an interval $J_1$ of length at most $\lceil |I_1|/c+1 \rceil \le C_d'$. Hence, $kA$ is contained in a translate of \[\{0,q_1,2q_1,\dots,(c-1)q_1\} + \{xcq_1:x\in kJ_1\} + kA',\]
which is in turn contained in a translate of 
\[\{0,q_1,2q_1,\dots,(c-1)q_1\} + C_d'\cdot 2k(A'-A')+kA'.\]
Therefore,
\[
|kA|\le |[0,c-1]q_1 + (2C_d'+1)kA'-2C_d'kA'| \le 2C_d' \cdot (5C_d'k)^{d-1} 
|A'| \le k^{d-1}C_d|A|,
\]
assuming that $C_d \ge 5^{d}C_d'^{d}$. 

The desired statement thus follows for $C_d = \max(5^d C_d'^{d}, C_d'^{2d-1},C_d'C_{d-1})$.
\end{proof}

We now come to our main structural result. As mentioned above, this roughly says that, for $A\subseteq [0,n-1]$ with $0\in A$ and $h$ at least a small power of $n$, there is a GAP $P$ such that $A$ is contained in $P$, while, for some constant $c>0$, $hA$ contains a proper translate of $chP$ (see Corollary \ref{cor:AP-doubling} for the exact statement). Since $hA$ is contained in $hP$, the latter may be viewed as an approximation for $hA$. 
As indicated in the statement of the lemma, the GAP $P$ and its dimension are essentially determined at the step of  slowest growth as we iteratively double $A$. 

\begin{lem} \label{lem:AP-from-doubling}
For any $\beta>1$, there exist positive integers $T$, $C$ and $C'$ depending only on $\beta$ such that the following holds for $n$ sufficiently large. 

Let $A$ be a subset of $[0,n-1]$ with $0\in A$ and let $h$ be a positive integer with $n\le h^{\beta}$. Let $d'$ be the smallest integer for which there exists $z$ with $T \le z\le\log(h/T)$ such that $|2^{z+1}A|\le2^{d'+1/2}|2^{z}A|$ and let $y$ be the smallest such integer $z$. 
Then there exists a GAP $Q$ of dimension $d\le d'$ such that
\begin{enumerate}
\item $Q$ is centered, $|Q|\le C|2^{y+1}A|$ and $A$ is contained in $2^{-(y+1)}Q$. 
\item The set $C 2^{y+1}A$ contains a translate of $Q$. 
\item $C'^{-1} h2^{-(y+1)}Q$ is proper and $h > 1000 C 2^{y+1}$. 
\end{enumerate}
\end{lem}

\begin{proof}
As $d'$ is the smallest integer for which there exists $z\in [T,\log (h/T)]$ such that $|2^{z+1}A| \le 2^{d'+1/2}|2^zA|$, we have that $|2^{u+1}A| > 2^{d'-1/2} |2^{u}A|$ for all $u\in [T, \log(h/T)]$. Hence, $$|(h/T)A|=|2^{\log(h/T)} A| = |2^T A|\prod_{u=T}^{\log(h/T)-1} \left(|2^{u+1}A|/|2^u A|\right)\ge 2^{(d'-1/2)(\log(h/T)-T)}$$ and, therefore, 
$$h^{1+\beta}/T \ge (h/T)n \ge |(h/T)A| \ge 2^{(d'-1/2)(\log(h/T)-T)},$$
where the first inequality uses $n\le h^{\beta}$ and the second uses $A \subseteq [0, n-1]$. Rearranging, we obtain $$(h/T2^T)^{d'-3/2-\beta}\le 2^{T(1+\beta)}T^{\beta}.$$ 
Note that the right-hand side is bounded in terms of $T$ and $\beta$. If $d' > 3/2+\beta$, then, for $n$  sufficiently large in terms of $\beta$ and $T$, since $h\ge n^{1/\beta}$, we would obtain a contradiction. Hence, we must have $d' \le 3/2+\beta<\beta+2$.  
In particular, $d'$ is upper bounded by a quantity depending only on $\beta$. 

By Lemma \ref{lem:bilu}, there exist $k$ and $\tilde{C}$ such that $2^{y+1}A$
is contained in a $2$-proper $k$-dimensional GAP $\tilde{Q}$ with
size at most $\tilde{C} |2^{y+1}A|$, where $\tilde{Q}=W\oplus Q$
with $Q$ being a GAP with dimension $d\le d'$ and $|Q|\ge\tilde{C}^{-1}|\tilde{Q}|$. 
Since $\tilde{Q}$ is $2$-proper, we thus also have $|W|\le\tilde{C}$. Moreover, $Q$ is also $2$-proper and, in particular, proper.  

Since $0\in A$, there exists a translate of $Q$ by $w\in W$ which contains $0$. By replacing $W$ with $W-w$ and $Q$ with $Q+w$, we can assume without loss of generality that this translate is equal to $Q$ and, therefore, that $0\in Q$. We show that $A$ must be contained in $2^{-y-1}Q$. Indeed, assume that $A$ contains an element $x$ outside $2^{-y-1}Q$. Note that $tx\in2^{y+1}A$ for
all positive integers $t \le 2^{y+1}$. 
For each $t\in [2^{y+1}]$, by the $2$-properness of $\tilde{Q}$, we can uniquely write $tx = w_t+q_t$, where $w_t \in W$ and $q_t \in Q$. Since $\tilde{Q}$ is $2$-proper and 
$$2w_t+2q_t=2tx=(t-1)x+(t+1)x=(w_{t-1}+w_{t+1})+(q_{t-1}+q_{t+1}),$$ 
we have $2w_t=w_{t-1}+w_{t+1}$ and $2q_t = q_{t-1}+q_{t+1}$ for all $1\le t<2^{y+1}$. Hence, since $q_0 = w_0 = 0$, we have $q_t =tq_1$ and $w_t = tw_1$ for all $1\le t\leq2^{y+1}$. If $w_1=0$, then, since $q_{2^{y+1}} \in Q$, we have $x = q_1 \in 2^{-y-1}Q$, as desired. If, instead, $w_1 \ne 0$, we have $w_t = tw_1 \ne 0$ for all $t\le 2^{y+1}$. Therefore, $|W| \ge 2^{y+1} \ge 2^{T}$. 
However, for $T$ chosen sufficiently large, this contradicts the bound $|W|\le\tilde{C}$. 

Hence, there exists a proper GAP $Q$ with dimension $d\le d'$ such that $0\in Q$, $|Q|\le \tilde{C} |2^{y+1}A|$ and 
$A$ is a subset of $2^{-(y+1)}Q$. By Claim \ref{claim:center}, $Q$ is  centered, that is,  $Q=\{\sum_{i=1}^{d}n_iq_i:n_i \in [a_i,b_i]\}$ with $a_i \le 0 \le b_i$ for all $i\in [d]$. We can also assume that the minimum width $\min(b_i-a_i+1)$ of $Q$ is at least $2^{y+1} \ge 2^{T+1}$. Otherwise, if $b_j-a_j < 2^{y+1}$, then, since $b_ja_j \le 0$, we have  $[2^{-(y+1)}a_j,2^{-(y+1)}b_j]\cap \mathbb{Z}=\{0\}$. 
Thus, letting 
$Q^*=\{\sum_{i\le d,i\ne j}n_iq_i:n_i \in [a_i,b_i]\}$, we see that $0\in Q^*$, $|Q^*|\le \tilde{C}|2^{y+1}A|$ and $A$ is a subset of $2^{-(y+1)}Q^*$. 

Let $\phi:Q\to \mathbb{Z}^{d}$ be the identification map. Note that $\phi(2^{y+1}A)$ is a dense subset of the box $\phi(Q)$.
By the first part of Corollary \ref{cor:non-reduced}, for $T$ chosen sufficiently large in $\tilde{C}$ and $d'$ (so that, in particular, the minimum width of $Q$ is sufficiently large in $\tilde{C}$ and $d'$), we can find $\hat{C} \ge \tilde{C}$ depending only on $\tilde{C}$ and $d'$ and a $d$-dimensional GAP $Q'$ of dimension $d$ in $\mathbb{Z}^d$ such that $Q'$ is contained in a translate of $\hat{C}\phi(Q)$, $Q'$ contains $\phi(Q)\cap \overline{\langle \phi(2^{y+1}A)\rangle}$ and $\hat{C}\phi(2^{y+1}A)$ contains a translate of $Q'$. Since $0\in A$, we have $\overline{\langle \phi(2^{y+1}A)\rangle} \supseteq \overline{\langle \phi(A)\rangle}$. Since $\phi(A)$ is a subset of $\phi(2^{-y-1}Q)\cap \overline{\langle \phi(A)\rangle}$ and $Q'$ contains $\phi(Q)\cap \overline {\langle \phi(2^{y+1}A)\rangle}$, we have that $\phi(A)$ is contained in $2^{-y-1}Q'$.  
Furthermore, since $Q'$ is contained in a translate of $\hat{C}\phi(Q)$, $\hat{C}^{-1}Q'$ is contained in a translate of $\phi(Q)$. Hence, as $Q'$ and $Q$ are proper, 
$\hat{C}^{-1}\phi^{-1}(Q')$ is proper. Finally, note that $|Q'|\le \hat{C}^{d}|Q|\le \hat{C}^d\tilde{C}|2^{y+1}A|$. Replacing $Q$ by $\phi^{-1}(Q')$, we may therefore assume that $Q$ is a GAP with the following properties: $0\in Q$ and so $Q$ is centered, $A$ is a subset of $2^{-y-1}Q$ and, for some $C$ depending only on $\tilde{C}$ and $d'$, $|Q|\le C|2^{y+1}A|$, $C^{-1}Q$ is proper and $C2^{y+1}A$ contains a translate of $Q$. Note crucially that $C$ only depends on $\tilde{C}$ and $d'$ but not on $T$. It therefore only remains to verify condition $3$ of the lemma, which we now turn to.

By taking $T$ sufficiently large in terms of $C$, we can guarantee that $h \geq 2^y T > 1000C2^{y+1}$. 
Let $C'$ be a sufficiently large constant depending on $\beta$ to be chosen later. If $C'^{-1}h < C^{-1}2^{y+1}$, then $C'^{-1}h2^{-(y+1)}Q$ is proper, since $C^{-1}Q$ is proper. Thus, we can assume that $C'^{-1}h \ge C^{-1}2^{y+1}$. Let $u\ge y+1$ be the largest integer such that $2^{u-y-1}C^{-1}Q$ is proper. Note that $u$ must exist since $C^{-1}Q$ is proper. If $C'^{-1}h2^{-(y+1)}Q$ is not proper, then, $C'^{-1}h > C^{-1}2^{u}$ so $u < \log h - \log C' + \log C$. 
By Corollary \ref{cor:nonproper-doub}, we therefore have that 
\[
|2^{k+u-y-1}C^{-1}Q|\le 2^{k(d-1)}C_d |2^{u-y-1}C^{-1}Q|.
\]
Since $A \subseteq 2^{-(y+1)} Q$, 
\[
|2^{k+u-y-1}2^{y+1}C^{-1}A|\le 2^{k(d-1)}C_d |2^{u-y-1}C^{-1}Q|.
\]
Furthermore, we have that a translate of $Q$ is contained in $C2^{y+1}A$,
so 
\[
|2^{k-\log C+u}A| = |2^{k+u-y-1}2^{y+1}C^{-1}A|\le2^{k(d-1)}C_d |2^{u}A|.
\]
Recall that, by the definition of $d'$, we have $|2^{x+1}A| > 2^{d'-1/2} |2^{x}A|$ for all $x\in [T, \log(h/T)]$. 
Moreover, we have that $u\ge y+1\ge T+1$.  
Thus, for $k\le\log(h/T)-u$, 
\[
|2^{k-\log C+u}A| = |2^{k-\log C} \cdot 2^{u}A| \ge 2^{(d'-1/2)(k - \log C)}|2^{u}A|,
\]
so we must have that, for all $1\le k \le \log(h/T)-u$, 
\[
k(d-1) + \log C_d \ge (d'-1/2)(k-\log C).
\]
This implies that 
\[
(\log(h/T)-u)/2 \le \log C_d+\log C(d'-1/2),
\]
so 
\[
u \ge \log h - \log T - 2(\log C_d+\log C(d'-1/2)).
\]
However, $u < \log h - \log C' + \log C$. Thus, since $d'<\beta+2$, by choosing $C'$ sufficiently large in terms of $C_d$, $C$, $\beta$ and $T$,  we arrive at a contradiction. Thus, $C'^{-1}h2^{-(y+1)}Q$ is proper. 
\end{proof}

The following definition, arising from  Lemma~\ref{lem:AP-from-doubling}, will be crucial in the proof of Theorem~\ref{thm:hom-AP-build}. 

\begin{defn}
Given positive integers $h$ and $n$  with $h \le n$ and a subset $A\subseteq [0,n-1]$ with $0\in A$, the \textit{$h$-dimension} of $A$ is the least dimension $d$ obtained from applying Lemma~\ref{lem:AP-from-doubling} with some $\beta>1$ such that $n\le h^{\beta}$. 
\end{defn}

Crucially, the proof of Lemma \ref{lem:AP-from-doubling}  yields that the $h$-dimension of any subset of $[0, n-1]$ with $n \leq h^{\beta}$ 
is bounded by a constant depending only on $\beta$.

We may now restate the conclusion of Lemma~\ref{lem:AP-from-doubling} succinctly in terms of the notion of $h$-dimension.

\begin{cor} \label{cor:AP-doubling}
For every $\beta > 1$, there exists a constant $c_\beta > 0$ such that the following holds. 
Let $A$ be a subset of $[0, n-1]$ with $0\in A$, let $h$ be a  positive integer such that $n\le h^\beta$ and let $d$ be the $h$-dimension of $A$. Then there exists a $d$-dimensional
GAP $P$ such that $A$ is contained in $P$ and $hA$ contains a proper translate of $c_{\beta}hP$. 
\end{cor}

\begin{proof} 
Let $Q$, $y$, $C$ and $C'$ be as given by applying Lemma~\ref{lem:AP-from-doubling} to $A$. 
Let $P=2^{-(y+1)}Q$. By Lemma \ref{lem:AP-from-doubling}, we have $A\subseteq P$, $C'^{-1}hP=C'^{-1}h2^{-(y+1)}Q$
is proper and $hA \supseteq \lfloor C^{-1}2^{-(y+1)}h\rfloor C2^{y+1}A$ contains a translate
of $\lfloor C^{-1}2^{-(y+1)}h\rfloor Q \supseteq \lfloor C^{-1}h/2\rfloor P$. Note that here we used that $h>1000C 2^{y+1}$, which follows from Item 3 of Lemma \ref{lem:AP-from-doubling}. Thus, $hA$ contains a proper translate of $\lfloor \min(C^{-1},C'^{-1})h/2\rfloor P \supseteq c_{\beta}hP$. 
\end{proof}

The following definition will also be important.

\begin{defn}\label{def:bounding-box}
Let $A$ be a finite set of natural numbers with $0\in A$ and let $d$ be a positive
integer. We define the \textit{$d$-bounding box $P_{d}(A)$} of $A$
to be the $d$-dimensional GAP containing $A$ with the smallest volume (breaking ties arbitrarily).
\end{defn}

With this definition in place, we can sum up the results of this section in the form we generally apply them.

\begin{lem} \label{lem:lower-hA}
Let $A$ be a subset of $[0,n-1]$ with $0\in A$ and let $h$ be a positive integer with $n\le h^{\beta}$. Assume that $n$ is sufficiently large in terms of $\beta$. If $A$ has $h$-dimension $d$, then 
\[
|hA| \gg_{\beta} h^d \textrm{Vol}(P_d(A))\ge h^{d}|P_{d}(A)|.
\]
Furthermore, the $h$-dimension of $A$ is at most $1+\beta$ and there exists a constant $c_\beta>0$ depending only on $\beta$ such that $c_\beta h P_d(A)$ and, consequently, $P_d(A)$ is proper. 
\end{lem}

\begin{proof}
Let $P$ be as in Corollary \ref{cor:AP-doubling}. Then, since $c_\beta hP$ is proper, for $n$ sufficiently large we have that $P$ is proper and  
$|P|=\textrm{Vol}(P)\ge\textrm{Vol}(P_{d}(A))$. Since $hA$ contains a translate of $c_{\beta}hP$ and $c_{\beta}hP$ is proper, Lemma \ref{lem:GAP-size} implies that 
\[
|hA| \gg_{\beta} h^{d}|P| \ge h^{d}\textrm{Vol}(P_d(A))\ge h^{d}|P_{d}(A)|.
\]
Therefore, since $h^d \ll_\beta |hA| \le hn \le h^{1+\beta}$, 
it follows immediately that if $n$ is sufficiently large in terms of $\beta$, then the $h$-dimension of $A$ is at most $1+\beta$. 

Finally, for an appropriate $c_\beta>0$ to be determined later, assume that $c_\beta hP_d(A)$ is not proper. Then, by Corollary \ref{cor:nonproper-doub}, if $t\le c_\beta h$ is the largest integer such that $tP_d(A)$ is proper, then
\[
|hP_d(A)| \le (h/t)^{d-1}C_d|tP_d(A)| \le (h/t)^{d-1} C_d t^d|P_d(A)| \le c_\beta C_d h^{d}\textrm{Vol}(P_d(A)).
\]
On the other hand, $|hP_d(A)|\ge |hA|\gg_\beta h^d \textrm{Vol}(P_d(A))$. This is a contradiction provided $c_\beta$ is  sufficiently small. Thus, we obtain that $c_\beta hP_d(A)$ is proper and, therefore, for $n$ sufficiently large, $P_d(A)$ is also proper. 
\end{proof}

\subsection{Non-proper GAPs}\label{subsec:non-proper}

In this subsection, we show that a homogeneous $d$-dimensional GAP $A$ contains either a proper homogeneous $d$-dimensional GAP whose volume is a constant fraction of the volume of $A$ or a homogeneous GAP whose dimension is at most $d-1$ and size is at least a constant fraction of $|A|$. Thus, by iterating, we arrive at a proper homogeneous GAP which is a subset of $A$ and whose size is at least a constant fraction of $|A|$. 

\begin{lem}
\label{lem:non-proper}
There exist positive constants $C_d,c_d$ depending only on $d$ such that the following holds. Let $A$ be a homogeneous $d$-dimensional
GAP. Then either $A$ contains a proper homogeneous GAP $Q$ of dimension at most $d$ and size at least $c_{d}\textrm{Vol}(A)$
or $A$ contains a homogeneous GAP $Q$ of dimension at most $d-1$ and size at least $c_{d}|A|$. Furthermore, $C_dQ$ contains a translate of $A$ and $\gcd(Q)=\gcd(A)$.
\end{lem}

\begin{proof}
We prove the result by induction on $d$, first noting that there is nothing to prove when $d = 1$. Suppose now that the result is true for $(d-1)$-dimensional GAPs and we would like to prove it for $d$-dimensional ones.

We first claim that, without loss of generality, we may assume that $0\in A$. Indeed, let $a\in A$, so that $0\in A-a$. Assume that $A-a$ contains a proper homogeneous GAP $Q$ of dimension at most $d$ and size at least $c_d\textrm{Vol}(A-a)=c_d\textrm{Vol}(A)$ or $A-a$ contains a homogeneous GAP $Q$ of dimension at most $d-1$ and size at least $c_d|A|$, where $C_dQ$ contains a translate of $A-a$ and $\gcd(Q)=\gcd(A-a)=\gcd(A)$. Since $A$ is homogeneous, $\gcd(A)|a$ and, hence, since  $\gcd(Q)=\gcd(A)$, the GAP $Q+a$ is also homogeneous. Thus, $Q+a$ satisfies all of the required properties.

Let $C_{d}',c_d'$ be the constants in Lemma \ref{lem:non-proper-d-1}. Let $A=\{\sum_{j=1}^{d} x_j q_j\,\,:\,\, x_j \in I_j\}$, where $I_j$ is an interval of length $w_j$. Let $h_0$ be the smallest positive integer such that $2^{-h_0}A$ has one of its widths smaller than $2C_{d}'$. We then let $h$ be the smallest positive integer, if it exists, which is at most $h_0$ and such that $2^{-h}A$ is proper, but $2^{-h+1}A$ is not. Otherwise, we set $h=h_0$. Note that $2^{-h}A$ has all widths at least $C_d'$. 
Let $A'=2^{-h}A$. We have that either $A'$ is proper or its minimum width is less than $2C_d'$.  

\vspace{2mm}
\noindent \textbf{Case 1:} $A'$ is proper and the widths of $A'$ are all at least $C_d'$. 

\vspace{2mm} \noindent
If $2^h \le C_d'$, then $A'$ has the required properties. Otherwise, assume that $2^h>C_d'$. By Lemma~\ref{lem:non-proper-d-1}, 
$C_d'A'$ contains a proper homogeneous $(d-1)$-dimensional GAP $B$ of size at least $c_d' |A'|$, where a translate of $A'$ is contained in $C_d'B$ and $\gcd(B) = \gcd(A') = \gcd(A)$. We have that $A\supseteq 2^h A' \supseteq \lfloor2^h C_d'^{-1}\rfloor C_d'A'$ contains the homogeneous $(d-1)$-dimensional GAP $2^{h-1}C_d'^{-1}B$. 
Furthermore, as $A'=2^{-h}A$, the set $A$ is contained in a translate of $2^{h+2}A'$ by Lemma \ref{lem:cGAP}, so $A$ is contained in a translate of $2^{h+2}C_d'B$. 
Thus, we have $|C_d'^{-1}2^{h-1}B| \gg_{d} |2^{h+2} C_d'B| \gg_{d} |A|$. By the induction hypothesis and the discussion preceding the lemma, 
$C_d'^{-1}2^{h-1}B$ contains a proper homogeneous GAP $Q$ of dimension at most $d-1$ with $|Q| \gg_d |C_d'^{-1}2^{h-1}B| \gg_d |A|$. Furthermore, $\gcd(Q)=\gcd(B)=\gcd(A)$ and, for $C_d$ sufficiently large, $C_dQ$ contains a translate of $2^{h+2}C_d'B$ and, hence, $A$.  

\vspace{2mm}
\noindent \textbf{Case 2:} The minimum width of $A'$ is less than $2C_d'$ and at least $C_d'$.

\vspace{2mm} \noindent
Let $A'=2^{-h}A=\{\sum_{i=1}^{d} n_iq_i:n_i \in [2^{-h}a_i,2^{-h}b_i]\}$. 
Without loss of generality, we may assume that $A'$ is not proper and that $C_d'\le 2^{-h}(b_1-a_1+1) < 2C_d'$. Let $B'= \{\sum_{i=2}^{d} n_iq_i:n_i \in [2^{-h}a_i,2^{-h}b_i]\}$. By the induction hypothesis, 
either $B'$ contains a proper homogeneous GAP $Q'$ of dimension at most $d-1$ and size at least $c_{d-1}\textrm{Vol}(B')$ or a proper homogeneous GAP $Q'$ of dimension at most $d-2$ and size at least $c_{d-1}|B'|$, where $\gcd(Q')=\gcd(B')$ and $C_{d-1}Q'$ contains a translate of $B'$. 

Let $A^*=Q'+[2^{-h}a_1,2^{-h}b_1]q_1$. Note that $A^*\subseteq A'$ and $C_{d-1}A^*$ contains a translate of $A'$. 
If $A^*$ is proper, then, since $C_{d-1}A^*$ is not proper, there exists $z\le C_{d-1}$ such that $zA^*$ is proper and $2zA^*$ is not proper. Let $\tilde{A} = zA^*$ and note that $\tilde{A}$ is homogeneous. By Lemma \ref{lem:non-proper-d-1}, $C_d'\tilde{A}$ contains a proper homogeneous $(d-1)$-dimensional GAP $B$ of size at least $c_d' |\tilde{A}|$, where $\tilde{A}$ is contained in a translate of $C_d'B$ and $\gcd(B)=\gcd(\tilde{A})=\gcd(A)$. We have that $A$ contains $2^h A^* \supseteq \lfloor 2^hz^{-1}\rfloor z\tilde{A} \supseteq 2^{h-1}C_{d-1}^{-1}\tilde{A}$, where we used that $z\le C_{d-1}$. Thus, $A$ contains the homogeneous $(d-1)$-dimensional GAP $(C_{d-1}C_d')^{-1}2^{h-1}B$. 
Furthermore, $A$ is contained in a translate of $2^{h+2}A'$ and, hence, in a translate of $2^{h+2}C_{d-1}A^*$ and $\tilde{A}=zA^*$ is contained in a translate of $C_d'B$, so $A$ is contained in a translate of  $C_{d-1}C_d'2^{h+2}B$. Thus, we have $|(C_{d-1}C_d')^{-1}2^{h-1}B| \gg_{d} |C_{d-1}C_d'2^{h+2} B| \gg_{d} |A|$. By the induction hypothesis and the remark preceding the lemma, $(C_{d-1}C_d')^{-1}2^{h-1}B$ contains a proper homogeneous GAP $Q$ of dimension at most $d-1$ with $|Q| \gg_d  |(C_{d-1}C_d')^{-1}2^{h-1}B| \gg_d |A|$ and $\gcd(Q)=\gcd(B)=\gcd(A)$. Furthermore, for $C_d$ sufficiently large, $C_dQ$ contains a translate of $C_{d-1}C_d'2^{h+2}B$ and, hence, $A$.  

Next, assume that $A^*$ is not proper. Since $Q'$ is proper, if $A^*$ is not proper, then there exists $q_1',q_2'\in Q'$ and $x_1,x_2\in [2^{-h}a_1,2^{-h}b_1]$ such that $q_1'+x_1q_1=q_2'+x_2q_1$. Thus, recalling that $2^{-h}(b_1-a_1)<2C_d'$, there is $\alpha\in [1,2C_d']$ such that $\alpha q_1 \in Q'-Q'$. Note that $A$ is contained in a translate of $2^{h+2}A'$ by Lemma \ref{lem:cGAP} and each element of $2^{h+2}(A'-\lceil 2^{-h}a_1\rceil q_1)$ can be written as the sum of an element of $2^{h+2}B'$ and $xq_1$ for $x\le 2^{h+2}(\lfloor 2^{-h}b_1\rfloor - \lceil 2^{-h}a_1\rceil)$. Since $\alpha q_1\in Q'-Q'$ for some $\alpha\in [1,2C_d']$, we can write $xq_1$ as the sum of an element in $[2C_d']q_1$ and an element in $2^{h+2}(\lfloor 2^{-h}b_1\rfloor - \lceil 2^{-h}a_1\rceil)(Q'-Q')\subseteq 2^{h+3}C_d'(Q'-Q')$.
Therefore, $A$ is contained in a translate of $[2C_d']q_1 + 2^{h+3}C_d'(Q'-Q') + 2^{h+2}B' \subseteq [2C_d']q_1 + 2^{h+4}C_d'(B'-B')$ 
and, in particular, $|2^{h}B'| \gg_d |A|$. 
By the induction hypothesis, either $2^{h}B'$ contains a proper homogeneous GAP $Q''$ of dimension at most $d-1$ and size at least $c_{d-1}\textrm{Vol}(2^{h}B')$ or a proper homogeneous GAP $Q''$ of dimension at most $d-2$ and size at least $c_{d-1}|2^{h}B'|$, where $\gcd(Q'')=\gcd(B')$ and $C_{d-1}Q''$ contains a translate of $2^{h}B'$. Let $Q = Q'' + \{0,1\}q_1$. 

If $Q$ is proper, then, since $A$ contains $2^{h}B'$, $A$ contains a proper translate of $Q$ of size at least $|Q| \gg_d |A|$. Note that for a GAP $B'$, we have that $-B'$ is a translate of $B'$, so $B'-B'$ is a translate of $2B'$. Hence, for $C_d \ge 64C_d'C_{d-1}$, we have that $C_dQ$ contains a translate of $[2C_d']q_1+2^{h+5}C_d'B'$, which contains a translate of $[2C_d']q_1+2^{h+4}C_d'(B'-B')$, 
which further contains a translate of $A$. We also have $\gcd(Q)=\gcd(q_1,\gcd(Q'')) = \gcd(q_1,B')=\gcd(A)$. Thus, $Q$ has the required properties. 

On the other hand, if $Q$ is not proper, then $q_1 \in Q''-Q''$ and, hence, as $A$ is contained in a translate of $[2C_d']q_1+2^{h+4}C_d'(B'-B')$, we have that 
$A$ is contained in a translate of $$(2C_d')(Q''-Q'')+2^{h+4}C_d'(B'-B')\subseteq ((2C_d')2^{h}+C_d'2^{h+4})(B'-B')\subseteq 2^{h+5}C_d'(B'-B'),$$ which is contained in a translate of $2^{h+6}C_d'B'$. Thus, recalling that $C_{d-1}Q''$ contains a translate of $2^{h}B'$, $A$ is contained in a translate of $64C_{d-1}C_d'Q'' \subseteq C_d Q''$. Furthermore, $\gcd(Q'')=\gcd(q_1,\gcd(Q'')) = \gcd(q_1,\gcd(B'))=\gcd(A)$, where we used that $q_1 \in Q''-Q''$. Hence, in this case, $Q''$ itself has the required properties. 
\end{proof}

\subsection{Stability under random sampling}\label{subsec:stability}

In this subsection, we define some notions of stability for subsets $A$ of $[0,n-1]$ and show that these properties are preserved for large random subsets of $A$. We will repeatedly use the fact that, by Lemma \ref{lem:lower-hA}, the $h$-dimension of a subset $A$ of $[0,n-1]$ with $0\in A$ for $h\ge n^{1/\beta}$ is bounded by $1+\beta$ when $n$ is sufficiently large in terms of $\beta$. 

\begin{defn} 
Let $x,\beta>1$ and let $A$ be a finite set of natural numbers with $0\in A$. For each positive integer $d$, let $P_{d}(A)$ be the $d$-bounding box of $A$. 
We say that $A$ is \textit{weakly-$(x,\beta)$-stable} if, 
for any $A'\subset A$ with $|A'|\ge|A|-x$ and $0\in A'$, we have that, for all $d\le 1+\beta$ and every GAP $P$ of dimension $d$ with differences at most $n^2$ and volume at most $\frac{3}{4}\textrm{Vol}(P_d(A))$, $A'$ is not contained in $P$.  
 When $0\notin A$, we say that $A$ is \textit{weakly-$(x,\beta)$-stable} if $A\cup \{0\}$ is. 
\end{defn}

The following observation will be important below.

\begin{lem}\label{lem:bound-on-box}
Let $\beta>1$, let $S$ be a subset of $[0,n-1]$ with $0\in S$ and, for $h \in [n^{1/\beta}, n]$, let $d$ be the $h$-dimension of $S$. Then, for $n$ sufficiently large in terms of $\beta$, the $d$-bounding box $P_d(S)$ has differences bounded above by $n^2$. 
\end{lem}

\begin{proof}
By Lemma \ref{lem:lower-hA}, we have $|hS|\gg_{\beta} h^d\textrm{Vol}(P_d(S))$. Assume, for the sake of contradiction, that $P_d(S)=\{\sum_{i=1}^{d}n_iq_i:n_i \in [a_i,b_i]\}$ and $q_1 > n^2$. Note that $hS\subseteq hP_d(S)\cap [0,h(n-1)]$. Furthermore, for each fixed $n_2,\dots,n_d$, there is at most one integer $n_1$ for which $n_1q_1+\sum_{i=2}^{d}n_iq_i \in [0,h(n-1)]$. Hence, $$|hP_d(S)\cap [0,h(n-1)]|\le \prod_{i=2}^{d}(b_i-a_i+1) \cdot h^{d-1} \le h^{d-1}\textrm{Vol}(P_d(S))/(b_1-a_1+1).$$
However, this contradicts the bound $|hP_d(S)\cap [0,h(n-1)]| \ge |hS| \gg_{\beta} h^d\textrm{Vol}(P_d(S))$ for $n$  sufficiently large. 
\end{proof} 

\begin{cor}\label{cor:weakly-stable}
Let $\beta>1$ and let $A$ be a weakly-$(x,\beta)$-stable subset of $[0,n-1]$. Then, for any subset $A'$ of $A$ of size at least $|A|-x$, any $h\in [n^{1/\beta},n]$ and $n$ sufficiently large in terms of $\beta$,  $\textrm{Vol}(P_d(A'\cup \{0\}))\ge \frac{3}{4}\textrm{Vol}(P_d(A))$,  where $d$ is the $h$-dimension of $A'\cup \{0\}$.
\end{cor}

\begin{proof}
Let $d$ be the $h$-dimension of $A'\cup \{0\}$, which is at most $1+\beta$ by Lemma \ref{lem:lower-hA}. By Lemma~\ref{lem:bound-on-box}, $P_d(A'\cup \{0\})$ has differences at most $n^2$. Since $A$ is weakly-$(x,\beta)$-stable, there is no GAP of dimension $d$ with differences at most $n^2$ and volume at most $\frac{3}{4}\textrm{Vol}(P_d(A))$ such that $P$ contains $A''$ for a subset $A''$ of $A \cup \{0\}$ with $|A''|\ge |A|-x$ and $0\in A''$. Hence, $\textrm{Vol}(P_d(A'\cup \{0\}))\ge \frac{3}{4}\textrm{Vol}(P_d(A))$.  
\end{proof}

The following lemma gives a useful property of weakly-$(x,\beta)$-stable sets. 

\begin{lem}
\label{lem:h-subset}
There is a constant $c_\beta>0$ such that if $A$ is a weakly-$(x,\beta)$-stable subset of $[0,n-1]$, then, for any subset $A'$ of $A$ of size at least $|A|-x$, any $h\in [n^{1/\beta},n]$ and $n$ sufficiently large in terms of $\beta$, 
\[
|h(A'\cup\{0\})| \ge c_{\beta} |hA|.
\]
\end{lem}

\begin{proof}
Let $d$ be the $h$-dimension of $A'\cup \{0\}$, which is at most $1+\beta$ by Lemma \ref{lem:lower-hA}. By the same lemma,
we have that, for some constant $c>0$ depending only on $\beta$, 
\[
|h(A'\cup\{0\})|\ge ch^{d}\textrm{Vol}(P_{d}(A'\cup\{0\}))\ge\frac{c}{2}h^{d}\textrm{Vol}(P_{d}(A))\ge\frac{c}{2}|hA|,
\]
where we used that $hA\subseteq hP_d(A)$ and, since $A$ is weakly $(x,\beta)$-stable, $\textrm{Vol}(P_d(A'\cup \{0\}))\ge \frac{3}{4}\textrm{Vol}(P_d(A))$ by Corollary \ref{cor:weakly-stable}.
\end{proof}

For a positive integer $d$, let $\phi_d$ be the identification map $\phi_d:P_d(A)\to \mathbb{Z}^{d}$. Weak stability tells us that the bounding box of any large subset $A'$ of $A$ is close in size to the bounding box of $A$, but later we will also need to control the subgroup $\langle \phi_d(A')\rangle$ spanned by $A'$ in $\mathbb{Z}^{d}$. 
Given a subset $A$ of $[0,n-1]$,  
let $\mathcal{D}_A$ be the set of $d$ for which there exists $h\in [n^{1/\beta},n]$ such that $d$ is the $h$-dimension of $A \cup \{0\}$. 
The next lemma shows that a weakly stable set $A$ contains a large subset $A'$ such that any large subset $A''$ of $A'$ spans the same subgroup of $\langle\phi_d(P_d(A))\rangle$ as $A'$ for all $d\in \mathcal{D}_A$. 

\begin{lem}\label{lem:stable-implies-span}
For any $\beta>1$, there exists $C_0\ge 1$ depending only on $\beta$ such that the following holds for $n$ sufficiently large in terms of $\beta$. Assume that $A$ is a weakly-$(x,\beta)$-stable subset of $[0,n-1]$. 
For each $d\in \mathcal{D}_A$, let $\phi_d$ be the identification map $\phi_d:P_d(A)\to \mathbb{Z}^{d}$. Then there exists a subset $A'$ of $A \cup \{0\}$ with $0\in A'$ and $|A'| \ge |A|-x/2$ such that, for all $d\in \mathcal{D}_A$ and any subset $A''$ of $A'$ with $0\in A''$ and $|A''|\ge |A'|-x/C_0$, $\langle \phi_d(A'')\rangle = \langle \phi_d(A')\rangle$. 
\end{lem}

\begin{proof}
Say that a subset $A'$ of $A \cup \{0\}$ is good if $0\in A'$ and, for all $d\in \mathcal{D}_A$ and any subset $A''$ of $A'$ with $|A''|\ge |A'|-x/C_0$ and $0\in A''$, $\langle \phi_d(A'')\rangle = \langle \phi_d(A')\rangle$. 

Let $A'_0 = A$. We iterate the following step. For $i \in [0,C_0/2]$, if $A'_i$ is good, then we output $A'=A'_i$. Otherwise, there exists $d\in \mathcal{D}_A$ and a subset $A''$ of $A'_i$ with $|A''|\ge |A'_i|-x/C_0$, $0\in A''$ and $\langle \phi_d(A'')\rangle \subsetneq \langle \phi_d(A'_i)\rangle$. Set $A'_{i+1} = A''$ and continue. We terminate when either $A'_i$ is good or we arrive at $i > C_0/2$. Observe that if the procedure terminates at iteration $i \le C_0/2$, then $A'_i$ is good and satisfies the desired property in the lemma statement. 

By Lemma \ref{lem:lower-hA}, $\max \mathcal{D}_A \le \beta + 1$. Thus, for each $i$, there exists $d\in \mathcal{D}_A$ such that the subgroup $\langle \phi_d(A'_i)\rangle$ has index at least $2^{i/(\beta+1)}$ in $\mathbb{Z}^{d}$. Let $\pi_j$ be the projection onto the $j$-th coordinate in $\mathbb{Z}^{d}$. Then there exists $j\le d$ such that  $\pi_j(\langle \phi_d(A'_i)\rangle)$ has index at least $2^{i/(d(\beta + 1))}$. In particular, for any box $B$ containing $0$, we have $|\langle \phi_d(A'_i)\rangle \cap B| \le |B|/2^{i/(d(\beta+1))}$. 

Assume that this procedure has not terminated by the $i$-th iteration. Let $d\in \mathcal{D}_A$ be such that the subgroup $\langle \phi_d(A'_i)\rangle$ has index at least $2^{i/(\beta+1)}$ in $\mathbb{Z}^{d}$. Since $d\in \mathcal{D}_A$, there exists $h \in [n^{1/\beta},n]$ such that $d$ is the $h$-dimension of $A$. 
Let $d'$ be the $h$-dimension of $A'_i$. By Lemma \ref{lem:lower-hA}, $d'$ is bounded in $\beta$ and 
\[
|h A'_i| \gg_\beta h^{d'} \textrm{Vol}(P_{d'}(A'_i)).
\]
Since $A$ is weakly-$(x,\beta)$-stable, Corollary~\ref{cor:weakly-stable} implies that $\textrm{Vol}(P_{d'}(A'_i)) \ge \frac{3}{4} \textrm{Vol}(P_{d'}(A))$ and so
\[
|h A'_i| \gg_\beta h^{d'} \textrm{Vol}(P_{d'}(A)).
\]
Furthermore, $A\subseteq P_{d'}(A)$, so $|h A| \le h^{d'} \textrm{Vol}(P_{d'}(A))$. Hence, we obtain that $|h A'_i| \gg_\beta |h A|$. But, again by Lemma \ref{lem:lower-hA}, 
\[
|h A| \gg_\beta h^{d} \textrm{Vol}(P_d(A)), 
\]
so that
\[
|h A'_i| \gg_\beta h^{d} \textrm{Vol}(P_d(A)).
\]
Since $h \phi_d(A'_i) \subseteq \langle \phi_d(A'_i)\rangle \cap h \phi_d(P_d(A))$, we have that 
\[
|\langle \phi_d(A'_i)\rangle \cap h \phi_d(P_d(A))| \gg_\beta h^{d} \textrm{Vol}(P_d(A)).
\]
Since we also have that 
\[
|\langle \phi_d(A'_i)\rangle \cap h \phi_d(P_d(A))| \ll_\beta |h \phi_d(P_d(A))| / 2^{i/(d(\beta+1))} \le h^{d} \textrm{Vol}(P_d(A)) / 2^{i/(d(\beta+1))},
\]
we get the bound $i\le d C_\beta \le C'_\beta$ for some constants $C_\beta,C'_\beta$ depending only on $\beta$. 
In particular, if the constant $C_0$ in the lemma statement satisfies $C_0>2C'_\beta$, then we arrive at a contradiction if the procedure has not terminated by the $C_0/2$-th iteration. Hence, for such a $C_0$, we can always find the desired subset $A'$ in the lemma statement. 
\end{proof}

Taking the lead from this lemma, we now define a notion of strong stability.

\begin{defn}
Let $\beta>1$ and let $C_0$ be the constant depending on $\beta$ in Lemma \ref{lem:stable-implies-span}. Let $A$ be a subset of $[0,n-1]$ with $0\in A$. For each positive integer
$d$, let $P_{d}(A)$ be the $d$-bounding box of $A$ and $\phi_d:P_d(A)\to \mathbb{Z}^{d}$ its identification map. 
We say that $A$ is \textit{strongly-$(x,\beta)$-stable} if it is weakly-$(x,\beta)$-stable and, 
for all $d\in \mathcal{D}_A$ and any $A'\subset A$ with $|A'|\ge|A|-x/C_0$ and $0\in A'$, we have that $\langle \phi_d(A')\rangle = \langle \phi_d(A)\rangle$. When $0\notin A$, we say that $A$ is \textit{strongly-$(x,\beta)$-stable} if $A\cup \{0\}$ is.
\end{defn}

Thus, Lemma \ref{lem:stable-implies-span} implies that a weakly-$(x,\beta)$-stable set $A$ has a subset $A'$ of size at least $|A|-x/2$ such that $A'$ is strongly-$(x/2,\beta)$-stable.

\begin{lem}
\label{lem:stable-subsets}For $\beta>1$, let $C_0$ be the constant depending on $\beta$ in Lemma \ref{lem:stable-implies-span}  and let $C\ge C_0$ be sufficiently large in terms of $\beta$. Let $A$ be a strongly-$(x,\beta)$-stable subset of $[0,n-1]$ of size $m$ with $0\in A$, where 
$n$ is sufficiently large in terms of $\beta$. Let $S$ be a random subset of $A$ of size $\alpha|A|$, where $\alpha x>C\log n$. 
Then the following claims hold: 
\begin{enumerate}
\item With probability at least $1-\exp(-\alpha x/16)$, the inequality $\textrm{Vol}(P_{d}(S\cup\{0\}))\ge\frac{3}{4}\textrm{Vol}(P_{d}(A))$ holds for all $d \in \mathcal{D}_{S}$.  
\item With probability at least $1-\exp(-\alpha x/(16C_0))$, the set $S$ is weakly-$(\frac{1}{2}\alpha x,\beta)$-stable and, furthermore, the following property holds. 
For each $d\in \mathcal{D}_A$, let $P_{d}(A)$ be the $d$-bounding box of $A$ and $\phi_d:P_d(A)\to \mathbb{Z}^{d}$ its identification map. Then, for any subset $S'$ of $S$ with $|S'|\ge |S|-\frac{1}{2}\alpha x/C_0$, 
$\langle \phi_d(S')\rangle = \langle \phi_d(A)\rangle$. In particular, $S$ is strongly-$(\frac{1}{2}\alpha x,\beta)$-stable. 
\end{enumerate}
\end{lem}

\begin{proof} We verify the two claims in turn.

\vspace{2mm}
\noindent {\bf Proof of 1.} Assume that $\textrm{Vol}(P_{d}(S\cup\{0\})) < \frac{3}{4}\textrm{Vol}(P_{d}(A))$ for some $d \in \mathcal{D}_{S}$, noting, by Lemma \ref{lem:lower-hA}, that $d\le 1+\beta$. Then there exists a GAP $P$ of dimension at most $d$ and volume less than $\frac{3}{4}\textrm{Vol}(P_d(A))$ such
that all elements of $A\setminus P$ are not contained in $S$. Since $A$ is strongly-$(x,\beta)$-stable, we have that $|A\setminus P|\ge x$. By a result of Hoeffding~\cite[Theorem 4]{Hoeff}, the probability that all elements
of $A\setminus P$ are not contained in $S$ is at most $\exp(-\alpha x/8)$. Note that there are at most $n^{4(1+\beta)}$ centered GAPs of dimension at most $1+\beta$ with differences at most $n^2$ and widths at most $n$. 
Therefore, by Lemma \ref{lem:bound-on-box} and the union bound, using the assumption that $\alpha x > C \log n$, we obtain
that the probability  $\textrm{Vol}(P_{d}(S\cup\{0\}))<\frac{3}{4}\textrm{Vol}(P_{d}(A))$ for some $d \in \mathcal{D}_{S}$ is at most $\exp(-\alpha x/16)/2$. 

\vspace{2mm}
\noindent {\bf Proof of 2.} 
Assume that we can remove at most $\frac{1}{2}\alpha x$ elements from $S$ to obtain $S'$ so that there is a GAP $P$ of dimension $d\le 1+\beta$ with differences at most $n^2$ and volume at most  $\frac{3}{4}\textrm{Vol}(P_{d}(S \cup \{0\}))$ that contains $S'\cup \{0\}$. 
In particular, there exists a GAP $P$ of dimension at most $1+\beta$ with differences at most $n^2$ and volume at most $\frac{3}{4}\textrm{Vol}(P_{d}(S \cup \{0\}))\le \frac{3}{4}\textrm{Vol}(P_{d}(A))$ such that $|S\cap(A\setminus P)|\le\frac{1}{2}\alpha x$. Since $A$ is strongly-$(x,\beta)$-stable, $|A\setminus P|\ge x$, so  Hoeffding's result again implies that the probability  $|S\cap(A\setminus P)|\le\frac{1}{2}\alpha x$ is at most $\exp(-\alpha x/8)$. By the union bound, taken over all $n^{4(1+\beta)}$ possible choices for the centered GAP $P$, the probability that we can remove at most $\frac{1}{2}\alpha x$ elements from $S$ to obtain $S'$ with $S'\cup \{0\}\subseteq P$ for some such $P$ is at most $\exp(-\alpha x/16)/2$. 

Let $h \in [n^{1/\beta},n]$ and let $d$ be the $h$-dimension of $A$. For any proper subgroup $\Gamma$ of $\langle \phi_d(A)\rangle$, 
since $A$ is strongly-$(x,\beta)$-stable, we have $|A\setminus \Gamma| \ge x/C_0$. Therefore, taking a union bound over the $n^{C_\beta}$ choices of possible subgroups spanned by elements of $\phi_d(P_d(A))$ and using that $\alpha x > C \log n$ for $C$ sufficiently large in $\beta$, the probability that we can remove at most $\frac{1}{2}\alpha x/C_0$ elements from $S$ to obtain $S'$ with $\langle \phi_d(S')\rangle$ a proper subgroup of $\langle \phi_d(A)\rangle$ is at most $\exp(-\alpha x/(16C_0))/2$. The required conclusion follows by combining the results of the two paragraphs.
\end{proof}

\subsection{Resilience and preprocessing}\label{subsec:processing}

In this short subsection, we describe a preprocessing step that outputs a stable subset of $A$, allowing us to apply the results of the previous subsection. We first define yet another notion of stability.

\begin{defn}
Given $\epsilon >0$, $\beta > 1$ and a subset $A$ of $[0,n-1]$, we say that $A$ is \textit{$(\epsilon,\beta)$-resilient} if, 
for any $d\le 1+\beta$ and any $A'\subseteq A$ of size at least $|A|/100$, we have $\textrm{Vol}(P_{d}(A'\cup \{0\}))\ge n^{-\epsilon}\textrm{Vol}(P_{d}(A \cup \{0\}))$. 
\end{defn}

We have the following consequence of Lemma~\ref{lem:lower-hA}.

\begin{cor}\label{cor:dim-d}
Let $\beta>1$, $C>0$ and let $\epsilon>0$ be sufficiently small in $\beta$ and $C$. Let $A \subseteq [0,n-1]$ be $(\epsilon,\beta)$-resilient with $0\in A$. Let $h\in [n^{1/\beta},n]$ and let $d$ be the $h$-dimension of $A$. Assume that $A$ is contained in a $d$-dimensional GAP $Q$ with identification map $\phi_Q$ and $|\phi_Q(Q)| \le C|P_d(A)|$. 
Then, for any subset $A'$ of $A$ with size at least $|A|/100$, $\phi_d(A')$ has dimension $d$.
\end{cor}

\begin{proof}
Assume that $\phi_Q(A')$ has dimension smaller than $d$. Then $\phi_Q(A')$ is contained in the intersection of a $(d-1)$-dimensional subspace $\Gamma$ and the box $P:=\phi_Q(Q)$ with widths $w_1,\dots,w_d$. 

Since $\Gamma$ has dimension $d-1$, there exists a basis vector $e_i\in \mathbb{Z}^d$ which is not contained in $\Gamma$. Hence, $\langle\phi_Q(A')\rangle$ intersects each translate of $\mathbb{Z}e_i$ in at most one point. Furthermore, the number of translates of $\mathbb{Z}e_i$ intersecting $hP$ is at most $h^{d-1}\prod_{j\ne i}w_j \le h^{d-1}|P|$. 
Hence, 
\[
|h(A'\cup \{0\})|\le |h\phi_Q(A'\cup \{0\})|\le |h\phi_Q(A) \cap \langle\phi_Q(A')\rangle| \le h^{d-1} |P|.
\]

On the other hand, by Lemma \ref{lem:lower-hA}, $d\le 1+\beta$ and, letting $d'$ be the $h$-dimension of $A' \cup \{0\}$, we have 
\[
|h(A'\cup \{0\})| \gg_{\beta} h^{d'} \textrm{Vol}(P_{d'}(A'\cup \{0\})). 
\]
Since $A$ is $(\epsilon,\beta)$-resilient, we have
\[
\textrm{Vol}(P_{d'}(A'\cup \{0\})) \ge n^{-\epsilon} \textrm{Vol}(P_{d'}(A)).
\]
Thus, 
\[
|h(A'\cup \{0\})|\gg_{\beta} n^{-\epsilon} h^{d'}\textrm{Vol}(P_{d'}(A)) \ge n^{-\epsilon} |hA|. 
\]
Again by Lemma \ref{lem:lower-hA}, we have $|hA|\gg_{\beta} h^{d}|P_d(A)|$, so 
\[
h^{d-1}|P| \gg_{\beta} n^{-\epsilon}h^d |P_d(A)| \ge C^{-1}n^{-\epsilon}h^d |P|.
\]
However, since $h\ge n^{1/\beta}$, this is a contradiction for $\epsilon$ sufficiently small. 
\end{proof}

The next lemma shows that we can replace a set $A$ with a large subset which is strongly stable and resilient. 

\begin{lem} \label{lem:reduction}
Let $\beta>1$, $\epsilon>0$ 
and let $A$ be a subset of $[0,n-1]$ of size $m$. 
Assume that $n \le m^{\beta}$ and $n$ is sufficiently large. Then there is a constant $c'>0$ depending only on $\epsilon$ and $\beta$ such that the following holds. 
For any positive integer $t$, there exists a subset $\tilde{A}$ of $A$ of size at least
$c'm - 100 \beta^2 t$ such that $\tilde{A}$ is both strongly-$(\frac{t}{\log m},\beta)$-stable
and $(\epsilon,\beta)$-resilient. 
\end{lem}

\begin{proof}
Assume that $A$ is not strongly-$(\frac{t}{\log m},\beta)$-stable
and $(\epsilon,\beta)$-resilient. We run the following process.

\vspace{2mm}

{\noindent \bf Step 1.} If $A$ is not weakly-$(\frac{2t}{\log m},\beta)$-stable, we can remove at most $\frac{2t}{\log m}$ elements from $A$ to obtain a subset $A'$ whose $d$-bounding box has volume at most a $3/4$-fraction of the $d$-bounding box of $A$ for some $d\le 1+\beta$. We replace $A$ by $A'$ and repeat this step until $A$ is weakly-$(\frac{2t}{\log m},\beta)$-stable, only then moving to Step 2. 

\vspace{2mm}

{\noindent \bf Step 2.} If $A$ is weakly-$(\frac{2t}{\log m},\beta)$-stable but not strongly-$(\frac{t}{\log m},\beta)$-stable, apply Lemma \ref{lem:stable-implies-span} to find a subset $A'$ of $A$ with $|A'| \ge |A|-\frac{t}{\log m}$ which is strongly-$(\frac{t}{\log m},\beta)$-stable. 
We replace $A$ by $A'$ and then move to Step 3. 

\vspace{2mm}

{\noindent \bf Step 3.} If $A$ is $(\epsilon,\beta)$-resilient, we terminate with the required set. If $A$ is not $(\epsilon,\beta)$-resilient, there is a subset $A'$ of
$A$ with size at least $|A|/100$ and $d\le 1+\beta$ such that $\textrm{Vol}(P_d(A'\cup \{0\}))< n^{-\epsilon}\textrm{Vol}(P_d(A\cup \{0\}))$. In this case, we replace $A$ by $A'$ and return to Step 1. 

\vspace{2mm}

Note that in each iteration of Step 1, the volume of the $d$-bounding box goes down by a factor of $3/4$. Since each $d$-bounding box of $A$ has size at most $n$ 
and there are at most $1+\beta$ choices for $d$, there are 
at most $(1+\log_{4/3}n)(1+\beta)$ such iterations. 
The number of iterations of Step 2 is bounded by the number of iterations of Step 1. Finally, the number of iterations of Step 3 is bounded by a constant in $\epsilon$ and $\beta$. Indeed, in each iteration of Step 3, for some $d\le 1+\beta$, we have that $\textrm{Vol}(P_d(A\cup \{0\}))$ decreases by a factor of at least $n^{\epsilon}$. Since $\textrm{Vol}(P_d(A\cup \{0\}))\le n$ in the first iteration,
 there can be at most $(1+\beta)/\epsilon$ many iterations of Step 3.  

Furthermore, in each iteration of Step 1 or Step 2, the size of the set decreases by at most $\frac{2t}{\log m}$, while in each iteration of Step 3, the size of the set decreases by at most a factor of $100$. Thus, the iterations must terminate at a strongly-$(\frac{t}{\log m},\beta)$-stable
and $(\epsilon,\beta)$-resilient set with size at least $c'm-2\cdot \frac{2t}{\log m}\cdot (\log_{4/3}n+1)(1+\beta) \ge c'm-100 \beta^2 t$, where $c' = 100^{-(1+\beta)/\epsilon}$ is a constant depending only on $\beta$ and $\epsilon$.
\end{proof}

We also have the following variant of Lemma \ref{lem:reduction} with a much larger $\tilde{A}$ if we do not require that $\tilde{A}$ is $(\epsilon,\beta)$-resilient. The proof is essentially identical to that of Lemma \ref{lem:reduction} and so is omitted.

\begin{lem}\label{lem:reduction-stable}
Let $\beta>1$
and let $A$ be a subset of $[0,n-1]$ of size $m$. 
Assume that $n \le m^{\beta}$ and $n$ is sufficiently large. Then, 
for any positive integer $t$, there exists a subset $\tilde{A}$ of $A$ of size at least
$m - 100 \beta^2 t$ such that $\tilde{A}$ is  strongly-$(\frac{t}{\log m},\beta)$-stable. 
\end{lem}

\subsection{Growing the set of subset sums}\label{subsec:growing-sum}

In this final subsection, we collect some simple results which will allow us to control the growth of a set of subset sums as we iteratively add elements to the set. Similar results can already be found in the work of Erd\H{o}s and Heilbronn~\cite{EH64} and Olson~\cite{Ols68} from the 1960s.

\begin{lem} \label{lem:period-addition}
Let $S$ be a finite set of integers and let
$a_{1},\dots,a_{k}$ be distinct integers. Then 
\[
|(S+a_1+\dots+a_k)\setminus S| \le \sum_{i=1}^{k}|(S+a_{i})\setminus S|.
\]
\end{lem}

\begin{proof}
We prove the lemma by induction on $k$. The statement is obvious for $k=1$. Assuming the statement is true for $k\le h$, we have, for $k=h+1$, that 
\begin{align*}
|(S+a_1+\dots+a_{h+1})\setminus S| & \le |(S+a_1+\dots+a_{h})\setminus S| + |(S+a_1+\dots+a_{h+1})\setminus (S+a_1+\dots+a_h)| \\
& =  |(S+a_1+\dots+a_{h})\setminus S| + |(S+a_{h+1})\setminus S| \\ & \le \sum_{i=1}^{h+1}|(S+a_{i})\setminus S|.
\end{align*}
Hence, the statement is true for all $k\ge 1$. 
\end{proof}

\begin{lem}
\label{lem:double-counting}
Let $S$ be a finite non-empty set of integers. Then
the set of $a\in\mathbb{Z}$ with $|(S+a)\setminus S| < \frac{|S|}{2}$
has size less than $2|S|$. 
\end{lem}

\begin{proof}
Consider the multiset $\{s-s':s,s' \in S\}$ of size $|S|^2$. If $|(S+a)\setminus S| < \frac{|S|}{2}$, then $a$ appears more than $|S|/2$ times in this multiset. Hence, there are fewer than  $|S|^2/(|S|/2)=2|S|$ such $a$. 
\end{proof}

\begin{lem}
\label{lem:growing-sum}Let $S$ and $A$ be finite non-empty sets of integers and 
let $k$ be such that $|kA| \ge 2|S|$. Then there exists $a\in A$
such that $|(S+a)\setminus S|\ge\frac{|S|}{2k}$. 
\end{lem}

\begin{proof}
Assume that $|(S+a)\setminus S|<\frac{|S|}{2k}$ for
all $a\in A$. By Lemma \ref{lem:period-addition}, this implies that 
$|(S+a)\setminus S|<\frac{|S|}{2}$ for all $a\in kA$. But then, by Lemma \ref{lem:double-counting},
we have $|kA| < 2|S|$, contradicting our assumption. 
\end{proof}

\section{Proof of Theorem \ref{thm:hom-AP-build}}

In this section, we give the proof of Theorem \ref{thm:hom-AP-build}
using the tools developed in the previous section. We first give an overview of the argument, which shares certain common features with the framework used in \cite{CFP}. First, we randomly partition $A$, or rather a large stable subset $\hat{A}$ of $A$, into $\ell$ sets $A_1,\dots,A_\ell$ of roughly equal size which, with high probability, inherit the relevant stability properties from $\hat{A}$. For each set $A_i$, 
we then find a subset $A_i'$ of size $cs/\ell$ for some positive constant $c$ such that  
\[
|\Sigma(A'_{i})|\gg_{\beta,\eta}\left|\frac{s}{\ell}(A\cup\{0\})\right|.
\]
Once this is achieved, we can obtain the desired homogeneous GAP by summing the sets $\Sigma(A'_i)$ and using Lemma \ref{lem:ell-box} and Corollary \ref{cor:non-reduced}. 

To show that we can find subsets $A'_i$ of each $A_i$ with  $|\Sigma(A'_i)|$ large, 
we consider an iterative procedure where we add in one element of $A_i$ at a time so as to maximize the growth of the set of subset sums at each step. After step $j$, we will have a subset $S_j$ of $A_i$ with $j$ elements removed and a set $\Sigma(j)$ consisting of the subset sums of the $j$ removed elements. We initialize with $S_{0}=A_{i}$ and $\Sigma(0)=\{0\}$. Then, at each step $j\ge1$, we pick an element $a_{j}\in S_{j-1}$ such that $|(\Sigma(j-1)+a_{j})\setminus\Sigma(j-1)|$
is maximized and let $S_{j}=S_{j-1}\setminus\{a_{j}\}$ and $\Sigma(j)=\Sigma(j-1)\cup(\Sigma(j-1)+a_{j})$. We run this iteration for $cs/\ell$ steps. 

In order to control the growth of $|\Sigma(j)|$ at each step, we appeal to Lemma \ref{lem:growing-sum}, which relates the growth of $|\Sigma(j)|$ to the size of the iterated sumsets of the available elements $S_{j-1}$. 
More concretely, 
$|\Sigma(j)|/|\Sigma(j-1)|$ will be at least $1+1/(2k_j)$, where $k_j$ is the smallest integer such that $|k_j(S_{j-1} \cup \{0\})| \ge 2|\Sigma(j-1)|$. Using that $|S_{j-1}| \ge |A_i| - j + 1$, we define certain numbers $t_h$ which give lower bounds on the sizes of $2^h (S_{j-1} \cup \{0\})$. 
In particular, when $|\Sigma(j-1)| \le t_h$, we have that $|\Sigma(j)|$ grows by a factor of at least $1+2^{-h-1}$. This allows us to bound the number of iterations where $|\Sigma(j)|$ lies in the interval $[t_h,t_{h+1}]$, as it must grow significantly in each such iteration. 
Combining this with estimates on $t_h$, we obtain the desired lower bound on $|\Sigma(cs/\ell)|$ at the end of our iteration.  

\begin{proof}[Proof of Theorem \ref{thm:hom-AP-build}]
By Lemma \ref{lem:reduction-stable} with $t=s\log m$,
we can replace $A$ by a subset $\hat{A}$ of size at least $m-100(4\beta/\eta)^2 s\log m$
which is strongly-$(s,4\beta/\eta)$-stable.

Let $\ell$ be a constant to be chosen later. Partition $\hat{A}$ randomly
into $\ell$ sets $A_{1},\dots,A_{\ell}$ of roughly equal size. Let $C_0$ be the constant depending only on $4\beta/\eta$ from Lemma \ref{lem:stable-implies-span}. By Lemma \ref{lem:stable-subsets}, there is a positive constant $C_1$ depending only on 
$4\beta/\eta$ such that if $s>C_1\ell \log n$, then, with probability at least $1 - \exp\left(-\frac{s}{32 C_0 \ell}\right)$, the following event $\mathcal{E}$ holds:  
\begin{itemize}
\item For all $i\in[\ell]$, $A_{i} \cup \{0\}$ is strongly-$(\frac{s}{2\ell},4\beta/\eta)$-stable.
\item For all $i\in[\ell]$ and all $d\in \mathcal{D}_{A_i}$,  $\textrm{Vol}(P_d(A_i\cup \{0\})) \ge \frac{3}{4} \textrm{Vol}(P_d(\hat{A}\cup \{0\}))$. 
\item For each $h$ such that $2^h \in [n^{\eta/4\beta},n]$, let $d$ be the $2^h$-dimension of $\hat{A}\cup \{0\}$. Let $P_{d}(\hat{A}\cup \{0\})$ be the $d$-bounding box of $\hat{A}\cup \{0\}$ and $\phi_d:P_d(\hat{A}\cup \{0\})\to \mathbb{Z}^{d}$ its identification map. Then, for all $i\in[\ell]$ and any subset $S$ of $A_i$ with $|S|\ge |A_i|-\frac{s}{2C_0\ell}$, $\langle \phi_d(S)\rangle = \langle \phi_d(\hat{A})\rangle$.  
\end{itemize}

We next show that, under the event $\mathcal{E}$, we can find a subset $A_{i}'$
of $A_{i}$ of size at most $\frac{s}{2C_0\ell}$ such that 
\[
|\Sigma(A'_{i})|\gg_{\beta,\eta}\left|\frac{s}{\ell}(\hat{A}\cup\{0\})\right|.
\]

We consider the following iterative process. Initialize $S_{0}=A_{i}$
and $\Sigma(0)=\{0\}$. At each step $j\ge1$, we pick an element
$a_{j}\in S_{j-1}$ such that $|(\Sigma(j-1)+a_{j})\setminus\Sigma(j-1)|$
is maximized. We then let $S_{j}=S_{j-1}\setminus\{a_{j}\}$ and
$\Sigma(j)=\Sigma(j-1)\cup(\Sigma(j-1)+a_{j})$. We run this iteration for $\frac{s}{2C_0\ell}$ steps. 

If, at step $j$, we let $k_{j}$ be
the smallest positive integer such that $|k_{j}(S_{j-1} \cup \{0\})| \ge 2|\Sigma(j-1)|$, then, by Lemma \ref{lem:growing-sum},
we have $|\Sigma(j)|\ge\left(1+\frac{1}{2k_{j}}\right)|\Sigma(j-1)|$.
For each positive integer $h$, let 
\[
t_{h}=\frac{1}{2}\min_{B\subseteq A_{i}:|B|\ge|A_{i}|-s/(2\ell)}|2^{h}(B\cup\{0\})|.
\]
If $|\Sigma(j-1)|\le t_{h}$ for some $j<\frac{s}{2C_0\ell}$, then, since $|S_{j-1}| \ge |A_i|-s/2\ell$, we have that $|2^h (S_{j-1} \cup \{0\})| \ge 2t_{h} \ge 2|\Sigma(j-1)|$. Thus, $k_{j}\le 2^h$, so that $|\Sigma(j)|\ge\left(1+\frac{1}{2^{h+1}}\right)|\Sigma(j-1)|$.
Therefore, using that $1+x > 2^{x}$ for $0<x<1$, the number of steps $j\le\frac{s}{2C_0\ell}$ where $|\Sigma(j)|\in[t_{h-1},t_{h})$ is at most $1 + 2^{h+1}\log\frac{t_{h}}{t_{h-1}}$. 

We claim that for $h> h_0 = \frac{\eta}{4} \log m$, $\frac{t_{h}}{t_{h-1}}$ is bounded by a constant depending only on $\eta$ and $\beta$. Indeed, let $d_{h-1}$ be the $2^{h-1}$-dimension
of $A_{i}\cup\{0\}$, which, since $2^{h-1} \ge m^{\eta/4} \ge n^{\eta/4\beta}$,  is bounded by a constant depending only on
$\beta$ and $\eta$ by Lemma \ref{lem:lower-hA}. We have 
\[
t_{h}\le|2^{h}(A_{i}\cup\{0\})|\le2^{hd_{h-1}}|P_{d_{h-1}}(A_{i}\cup\{0\})|.
\]
On the other hand, under the event $\mathcal{E}$, Lemmas \ref{lem:h-subset} and  \ref{lem:lower-hA} imply that
\[
t_{h-1} \gg_{\beta,\eta} |2^{h-1}(A_i\cup\{0\})| \gg_{\beta,\eta} 2^{(h-1)d_{h-1}}|P_{d_{h-1}}(A_{i}\cup\{0\})|.
\]
Thus, 
\[
\frac{t_{h}}{t_{h-1}} \ll_{\beta,\eta} 2^{d_{h-1}} \ll_{\beta,\eta} 1.
\]

Now let $h_{*}$ be such that $\left|\Sigma\left(\frac{s}{2C_0\ell}\right)\right|\in[t_{h_{*}},t_{h_{*}+1})$.
Then we have that
\[
\sum_{0 \le h \le h_{*}+1}\left(1+2^{h+1}\log\frac{t_{h}}{t_{h-1}}\right)\ge\frac{s}{2C_0\ell}.
\]
However, by the claim above, we see that
\begin{align*}
\sum_{0 \le h \le h_{*}+1}\left(1+2^{h+1}\log\frac{t_{h}}{t_{h-1}}\right) &\le (2+h_*) + 2^{h_0+1} \log t_{h_0} + \sum_{h_0<h\le h_*+1} 2^{h+1}\log\frac{t_{h}}{t_{h-1}} \\
&\ll_{\beta,\eta} h_* + 2^{(\eta/4)\log m} \cdot 4\log n + 2^{h_*},
\end{align*}
where we used that $t_{h_0} \le 2^{h_0} n \le n^2$. 
Hence, there is a constant $c(\beta,\eta)>0$ such that $$2^{h_*}+h_*+2^{(\eta/4)\log m}\cdot 4\log n \ge c(\beta,\eta)\frac{s}{2C_0\ell}.$$
Therefore, since $2^{h_*}\ge h_*$, 
\begin{equation}
2^{h_{*}} \ge \frac{c(\beta,\eta)}{2}\left(\frac{s}{2C_0\ell}\right) - 2m^{\eta/4}\log n \ge \frac{c({\beta,\eta})}{8C_0} \frac{s}{\ell} \label{eq:bound-h*}
\end{equation}
for $n$ sufficiently large, where we used that $s \ge m^{\eta}$ and $n\le m^\beta$. 

Hence, with $d_{h_*}$ being the $2^{h_*}$-dimension of $A_i\cup \{0\}$, Lemmas \ref{lem:h-subset} and  \ref{lem:lower-hA} imply that
$$
\left|\Sigma\left(\frac{s}{2C_0\ell}\right)\right| 
\ge t_{h_{*}} 
\gg_{\beta,\eta} 2^{h_*d_{h_*}}|P_{d_{h_*}}(A_i\cup \{0\})|
\gg_{\beta,\eta} \left|\frac{s}{\ell}(\hat{A} \cup \{0\})\right|,
$$
where, in the last inequality, we used \eqref{eq:bound-h*} to conclude that 
$$\left|\frac{s}{\ell}(\hat{A} \cup \{0\})\right| \le  \left(\frac{s}{\ell}\right)^{d_{h_*}} |P_{d_{h_*}}(\hat{A} \cup \{0\})| \ll_{\beta,\eta} 2^{h_*d_{h_*}} |P_{d_{h_*}}(\hat{A} \cup \{0\})|\ll_{\beta,\eta}2^{h_*d_{h_*}}|P_{d_{h_*}}(A_i\cup \{0\})|,$$ 
since, under the event $\mathcal{E}$, 
\[
|P_{d_{h_*}}(A_i\cup \{0\})|=\textrm{Vol}(P_{d_{h_*}}(A_i\cup \{0\})) \ge \frac{3}{4} \textrm{Vol}(P_{d_{h_*}}(\hat{A}\cup \{0\})) \ge \frac{3}{4} |P_{d_{h_*}}(\hat{A}\cup \{0\})|.
\]
Thus, there exists a subset $A_{i}'$ of $A_{i}$ of size at most
$\frac{s}{2C_0\ell}$ such that 
\[
|\Sigma(A'_{i})| \gg_{\beta,\eta} \left|\frac{s}{\ell}(\hat{A}\cup\{0\})\right|.
\]

By Lemma \ref{lem:lower-hA}, 
letting $d$ be the $\frac{s}{\ell}$-dimension
of $\hat{A} \cup\{0\}$, then $|h (\hat{A} \cup\{0\})| \gg_{\beta,\eta} h^d|P_d(\hat{A} \cup\{0\})|$, where $h=\frac{s}{\ell}$. 
Let $\tilde{P}=P_d(\hat{A} \cup\{0\})$. By Claim \ref{claim:center} and Lemma \ref{lem:lower-hA}, we can assume that $\tilde{P}$ is a centered GAP and that $\tilde{c}h\tilde{P}$ is proper for some $\tilde{c}$ depending only on $\beta$ and $\eta$. Let $\phi$ be the identification map $\phi:\tilde{P} \to \mathbb{Z}^{d}$. We then have that $h\phi(\hat{A}\cup \{0\})$ is a subset of $h\phi(\tilde{P})$ with $|h\phi(\hat{A}\cup \{0\})|\gg_{\beta,\eta} |h\phi(\tilde{P})|$ and $0\in h\phi(\hat{A}\cup \{0\})$. 
By the first claim in Corollary \ref{cor:non-reduced}, we have that there exists a $d$-dimensional centered GAP $\overline{Q}$ of dimension $d$ in $\mathbb{Z}^{d}$ with the following properties:
\begin{itemize}
    \item $\overline{Q} \supseteq \langle h\phi(\hat{A}\cup \{0\})\rangle \cap h\phi(\tilde{P})$,  
    \item $\overline{Q}$ is contained in a translate of $C_{\beta,\eta} h\phi(\tilde{P})$, 
    \item $C_{\beta,\eta}h\phi(\hat{A} \cup\{0\})$ contains a translate of $\overline{Q}$ (and so $h\phi(\hat{A} \cup\{0\})$ is reduced in $\overline{Q}$), 
    \item $|\overline{Q}|\le C_{\beta,\eta}|h\phi(\hat{A} \cup\{0\})|$.
\end{itemize}
Since $h\phi(\hat{A}\cup \{0\})\subseteq \overline{Q}$ and $\phi(\hat{A}\cup \{0\})$ is a subset of $\mathbb{Z}^d$, we have $\phi(\hat{A}\cup \{0\}) \subseteq h^{-1}\overline{Q}$. Let $Q=h^{-1}\overline{Q}$. We then have that $Q\supseteq \langle h\phi(\hat{A}\cup \{0\})\rangle \cap \phi(\tilde{P})$, $Q$ is contained in a translate of $C_{\beta,\eta}\phi(\tilde{P})$ and $\phi(\hat{A}\cup \{0\})$ is reduced in $Q$. Furthermore, since $Q$ is contained in a translate of $C_{\beta,\eta}\phi(\tilde{P})$, it follows that $\bar{c}hQ$ is contained in a translate of $\phi(\tilde{c}h\tilde{P})$ for some $\bar{c}$ depending only on $\beta$ and $\eta$ and thus, as $\bar{c}hQ$ is a proper GAP in $\mathbb{Z}^d$ and $\tilde{c}h\tilde{P}$ is proper, 
$\phi^{-1}(\bar{c} hQ)$ is also proper. 

For all $i\in[\ell]$, we have 
\[
\Sigma(A_{i}')\subseteq \phi^{-1}(\frac{s}{2\ell}\phi(\tilde{P})\cap \langle \phi(A_i')\rangle)
\]
and 
\begin{equation}
|\Sigma(A_{i}')| \gg_{\beta,\eta}\left|\frac{s}{\ell}(\hat{A}\cup\{0\})\right| \gg_{\beta,\eta} (s/\ell)^{d}|\tilde{P}|. \label{eq:const-frac}
\end{equation}
Thus, $\langle \phi(A_i')\rangle$ must have index bounded by a constant in $\beta$ and $\eta$ in $\mathbb{Z}^{d}$. Under the event $\mathcal{E}$, since $|A_i'|\le \frac{s}{2C_0\ell}$, we have that $\langle \phi(A_i \setminus A_i')\rangle = \langle \phi(\hat{A})\rangle$. Thus, by greedily choosing elements of $A_i\setminus A_i'$, we obtain a subset $A_i''\subseteq A_i\setminus A_i'$ of size bounded in $\beta$ and $\eta$ such that $\langle \phi(A_i''\cup A_i')\rangle = \langle \phi(\hat{A})\rangle$. Note that $A_i'\cup A_i''\subseteq \hat{A}$ and $|A_i'|\le \frac{s}{2C_0\ell}$, so $\Sigma(A_i'\cup A_i'')\subseteq (\frac{s}{2C_0\ell}+C_{\beta,\eta})(\hat{A}\cup \{0\})$. Then, for $T_i = \phi(\Sigma(A_i' \cup A_i''))$, we have that $T_i\subseteq \frac{s}{2\ell}Q$ and $T_i$ is reduced in $\frac{s}{2\ell}Q$, since $\langle \phi(A_i'\cup A_i'')\rangle =\langle \phi(\hat{A})\rangle$ and $\phi(\hat{A})$ is reduced in $Q$. Furthermore, by (\ref{eq:const-frac}) and since $|Q|\ll_{\beta,\eta}|\phi(P)|$, $T_{i}$ occupies a constant fraction $c'_{\beta,\eta}$ of
$\frac{s}{2\ell}Q$. 
By Lemma \ref{lem:ell-box}, for $\ell$ sufficiently large in terms of $\beta$ and $\eta$, the sum of the sets $T_{i}$ contains a translate of $\gamma\ell\frac{s}{2\ell}Q=\frac{\gamma s}{2}Q$ for some constant $\gamma$ depending only on $c'_{\beta,\eta}$ and $d$.
Hence, $\Sigma(T_1\cup\dots\cup T_{\ell})$ contains a translate
of $\frac{\gamma s}{2}Q$ by an element of $\Sigma(\phi(\hat{A}))\in \langle \phi(\hat{A})\rangle$. 

Since $\phi^{-1}(\bar{c} h Q)$ is proper for some $\bar{c}$ depending on $\beta$ and $\eta$, we obtain that there is a GAP $P:=\phi^{-1}(Q)$ and a subset $A'=\bigcup_{i=1}^{\ell}A_i'\cup A_i''$ of $\hat{A}$ of size at most $s$ such that $\hat{A}\cup \{0\}$ is contained in $P$ and $\Sigma(A')$ contains a homogeneous and proper translate of $csP$, where $c>0$ depends only on $\beta$ and $\eta$. 
\end{proof}

As we will need it for the proof of Theorem~\ref{thm:hom-AP}, we now record a variant of Theorem~\ref{thm:hom-AP-build} where $\hat{A}$ is explicitly stable and resilient. We omit the proof, which is the same as that above, except that we  apply Lemma~\ref{lem:reduction} rather than Lemma~\ref{lem:reduction-stable} at the outset.

\begin{thm} \label{thm:hom-AP-build2} 
For any $\beta>1$, $\epsilon > 0$ and $0<\eta<1$, there are positive constants $c$ and $d$ such that the following holds. Let $A$ be a subset of $[n]$ of size $m$ with $n\le m^{\beta}$ and let $s\in[m^{\eta},cm/\log m]$. Then there exists a subset $\hat{A}$ of $A$ of size at least $cm$ which is both strongly-$(s,\beta)$-stable
and $(\epsilon,\beta)$-resilient, a proper GAP $P$ of dimension at most $d$ such
that $\hat{A} \cup \{0\}$ is contained in $P$ and a subset
$A'$ of $\hat{A}$ of size at most $s$ such that $\Sigma(A')$
contains a homogeneous translate of $csP$, where $csP$ is proper. 
\end{thm}

\section{Convex geometry and subset sums}

In this section, we show that we can approximate the set of subset sums of a set $A$ by a certain convex polytope and collect several useful properties of this polytope. In the next section, we will then combine the results of this section with Theorem~\ref{thm:hom-AP-build2} to prove Theorem~\ref{thm:hom-AP}.

\begin{defn}
Given a finite subset $A$ of $\mathbb{Z}^{d}$, we define the zonotope ${\cal Z}_{A}$ to be the Minkowski sum of the segments $[0,1]\cdot a$ with $a\in A$. 
\end{defn}

\begin{lem} \label{lem:sampling}
Let $A$ be a subset of a box $Q$ in $\mathbb{Z}^{d}$ with widths $w_{1},\dots,w_{d}$ and $0\in A$ and let ${\cal Z}_{A}$ be the corresponding zonotope.
Then, for any $z\in{\cal Z}_{A}$, there exists a subset sum $s(z)$ of $A$ such that  $|z_{i}-s(z)_{i}|\le \sqrt{d|A|}w_{i}$ for all $i\le d$. 
\end{lem}

\begin{proof}
Since $z\in{\cal Z}_{A}$, we can write $z=\sum_{a\in A}z_{a}a$,
where the coefficients $z_{a}$ are in the interval $[0,1]$. 
For each $a$, let $Z_a$ be the random variable which is $1$ with probability $z_a$ and $0$ otherwise, with each $Z_a$ independent of all others. Consider also $Z = \sum_{a\in A} Z_a a$, noting that $\mathbb{E}[Z]=z$. 

We next compute the variance of the $i^{\textrm{th}}$ coordinate of $Z$, obtaining 
\[
\mathbb{E}\left[\left(Z-z\right)_{i}^{2}\right] = \mathbb{E}\left[\left(\sum_{a\in A}(Z_a-z_a)a_i\right)^{2}\right].
\]
Note that $|a_i|\le w_i$ from our assumption that $0\in A\subseteq Q$. By independence of the zero-mean random variables $(Z_a-z_a)a_i$, we have that
\begin{align*}
 \mathbb{E}\left[\left(\sum_{a\in A}(Z_a-z_a)a_i\right)^{2}\right] &= \sum_{a\in A} a_i^2 \mathbb{E}[(Z_a-z_a)^2] =a_i^2 \sum_{a\in A}  z_a(1-z_a) \le w_i^2 |A|/4. 
\end{align*}
Thus, by Chebyshev's inequality, we have 
\[
\mathbb{P}\left(\left|Z_i - z_i\right| \ge \sqrt{d|A|}w_i\right) \le \frac{1}{4d}.
\]
Hence, by the union bound, with probability at least $1-d \cdot \frac{1}{4d}=3/4$, we have that, for all $i\le d$, 
\[
\left|Z_{i}-z_{i}\right|\le\sqrt{d|A|}w_{i}.
\]
Since $Z\in\Sigma(A)$, we have arrived at the desired conclusion. 
\end{proof}

Given a subset $A$ of $\mathbb{Z}^d$, the \emph{dimension} of $A$ is the dimension of the span $\langle A\rangle$ in $\mathbb{R}^d$, while the \emph{affine dimension} of $A$ is the dimension of $A - a$ for any $a  \in A$. The next lemma says that any subset of $\mathbb{Z}^{d}$ with affine dimension $d$ contains a simplex of large volume.

\begin{lem}
\label{lem:simplex-volume}In any set $A$ of $m$ distinct integer
points in $\mathbb{Z}^{d}$ with affine dimension $d$, there exist
$d+1$ points such that the simplex spanned by these points has volume at least $c_{d}m$. 
\end{lem}

\begin{proof}
We first claim that if ${\cal P}$ is a simplex with maximum volume spanned by $d+1$ points of $A$, then $A$ can be covered by a copy of $2{\cal P}$. Indeed, consider the $d+1$ hyperplanes $H_x$ going through a vertex $x$ of ${\cal P}$ parallel to the face of ${\cal P}$ not containing $x$. Every point of $A$ must lie to the same side of the hyperplane $H_x$ as ${\cal P}$, as otherwise that point together with the vertices of ${\cal P}$ other than $x$ would define a simplex with larger volume than ${\cal P}$. Let ${\cal H}_x$ be the closed half-space containing ${\cal P}$ adjacent to $H_x$. The intersection of the half-spaces ${\cal H}_x$ defines a simplex $\tilde{\cal P}$ isomorphic to $2{\cal P}$, whose vertices are the reflections of each vertex $x$ of ${\cal P}$ about the face of ${\cal P}$ not containing $x$. Since $A$ is a subset of ${\cal H}_x$ for each $x$ in ${\cal P}$, $A$ is also a subset of their intersection $\tilde{\cal P}$, proving the desired claim.

By an old result of Blichfeldt \cite{Bl} (see also \cite{Bar}), the volume of the convex body spanned by a set of $m$ distinct integer points in $\mathbb{Z}^{d}$ with affine dimension $d$ is at least $\frac{1}{d!}(m-d)\ge \frac{1}{(d+1)!}m$, where we used that $m$ must be at least $d+1$. Combining this observation with the above claim, we obtain that $\textrm{vol}(2 {\cal P}) \ge \frac{1}{(d+1)!}m$ for a maximum volume simplex ${\cal P}$ spanned by points in $A$. In particular, ${\cal P}$ has volume at least $\frac{1}{2^d(d+1)!} m \ge c_{d}m$, as desired. 
\end{proof}

We now use this result to derive a lower bound on the volume of the zonotope ${\cal Z}_{A}$ associated with a set $A \subset \mathbb{Z}^d$.

\begin{lem}
\label{lem:zonotope-volume}Let $0<c<1/3$ and suppose $A \subset \mathbb{Z}^d$ has the property that every subset
of $A$ of size at least $c|A|$ has  dimension $d$. Then the volume of the zonotope
${\cal Z}_{A}$ is at least $c'|A|^{d+1}$, where $c'>0$ depends only on $c$ and $d$. 
\end{lem}

\begin{proof}
By Lemma \ref{lem:simplex-volume}, there exist $d+1$ points in
$A\cup \{0\}$ spanning a simplex with volume at least $c_{d}m$. Let $A_{1}$ be obtained
from $A$ by removing these points. 
We then repeat this process, stopping only when the dimension of the remaining points is less than $d$. 
By assumption, we can repeat this process for at least
$(1-c)|A|/(d+1)-1$ steps. This yields $(1-c)|A|/(d+1)-1$ simplices, each of volume at least $c_{d}c|A|$, such that ${\cal Z}_{A}$ contains the Minkowski sum of these simplices. By the Brunn--Minkowski inequality, we
see that the volume of ${\cal Z}_{A}$ is at least 
\[
\left(\left(\frac{(1-c)|A|}{d+1}-1\right)\cdot\left(c_{d}c|A|\right)^{1/d}\right)^{d}\ge c'|A|^{d+1}
\]
for an appropriate $c'>0$, as required.
\end{proof}

The final ingredient we will need is the following result of Tao and Vu~\cite[Theorem~3.36]{TV}. 
Recall that a subset $A$ of $\mathbb{R}^{d}$ is \emph{symmetric} if $A=-A$. Moreover, a \emph{lattice of rank $r$} in $\mathbb{R}^{d}$ is a discrete additive subgroup of $\mathbb{R}^{d}$ generated by $r$ linearly independent vectors. 

\begin{lem}\label{lem:TV-convexbody}
Let $B$ be a convex symmetric body in $\mathbb{R}^d$ and let $\Gamma$ be a lattice in $\mathbb{R}^d$ of rank $r$. Then there exists an $r$-tuple $w=(w_1,\dots,w_r)\in \Gamma^r$ of linearly independent vectors in $\Gamma$ and an $r$-tuple $N=(N_1,\dots,N_r)$ of positive integers such that 
\[
(r^{-2r} B)\cap \Gamma \subseteq (-N,N)\cdot w \subseteq B\cap \Gamma \subseteq (-r^{2r}N,r^{2r}N)\cdot w,
\]
where $(-N,N) \cdot w = \{\sum_{i=1}^{r}x_iw_i\,\,:\,\, x_i \in (-N_i,N_i) \cap \mathbb{Z}\}$. 
\end{lem}

We now come to the main result of this section. This says that given a dense subset $A$ of a box $Q$ in $\mathbb{Z}^d$, if we add a box $Q'$ that is not too large to $\Sigma(A)$, we cover all integer points in a neighborhood of the zonotope $\mathcal{Z}_A$. This then allows us to show the existence of a large GAP inside $\Sigma(A)+Q'$. 

\begin{lem} \label{lem:build-GAP}
For any positive integer $d$ and any $0<c<1/3$, there exist $c'', c'''>0$ such that the following holds. Let $A$ be a subset of a box $Q$ in $\mathbb{Z}^{d}$ with widths
$w_{1},\dots,w_{d}$ and $0\in A$ such that the dimension
of any subset of $A$ of size at least $c|A|$ is $d$. 
Let $Q'$ be a box in $\mathbb{Z}^{d}$ with widths $w_{1}',\dots,w_{d}'$ such that $Q'$ is symmetric around $0$ and $w_{i}'\ge 8\sqrt{d|A|}w_{i}$. 
Then $\Sigma(A)+Q'$ contains all integer points in ${\cal Z}_{A} + \mathrm{conv}(c''Q')$.
Furthermore, $\Sigma(A)+Q'$ contains a translate of a GAP $P$ of size at least $c'''\max(|Q'|,|A|^{d+1})$ such that the affine span of $P$ is $\mathbb{R}^{d}$ 
and $P$ contains  $c'''Q'$ and $c'''\tilde{\mathcal{Z}} \cap \mathbb{Z}^d$, where $\tilde{\mathcal{Z}}$ is a translate of the zonotope of a subset $A^*$ of $A$ of size at least $|A|-2^d$. 
\end{lem}

\begin{proof} 
From Lemma \ref{lem:sampling}, we have that for all points $z$ in
${\cal Z}_{A}$, there exists a point $\tilde{z}\in\Sigma(A)$ such
that 
\[
|z_{i}-\tilde{z}_{i}|\le\sqrt{d|A|}w_{i}.
\]
Thus, provided $c'' \leq 1/4$, for any $y=z+t$ with $t\in \textrm{conv}(c'' Q')$ and $y\in \mathbb{Z}^{d}$, we can find $\tilde{z} \in \Sigma(A)$ with
\[
|y_{i}-\tilde{z}_{i}|\le \sqrt{d|A|}w_{i} + c'' w_i' \le 3w_i'/8.
\]
Hence, 
\[
y \in\Sigma(A)+Q'.
\]

Since any subset of $\mathbb{Z}_2^d$ of size at least $d+1$ contains a non-empty subset with zero sum, we can iteratively remove non-empty subsets of $A$ with zero sum until we are left with at most $d+1$ elements. Hence, there is a subset $A^*$ of $A$ with size at least $|A|-(d+1)$ where $\sum_{a\in A^*}a/2 \in \mathbb{Z}^d$. But ${\cal Z}_{A^*} - \sum_{a\in A^*}a/2$ is symmetric about $0$, so that ${\cal Z}_{A^*}$ has an integer translate $\tilde{\cal Z}$ which is symmetric around $0$. Furthermore, by Lemma \ref{lem:zonotope-volume}, the volume of ${\cal Z}_{A^*}$ and, hence, that of $\tilde{\cal Z}$ is at least $c'|A^*|^{d+1} \ge \tilde{c}'|A|^{d+1}$.  

By Lemma \ref{lem:TV-convexbody}, 
there is a GAP $P$ such that $P\subseteq(\textrm{conv}(c''Q')+\tilde{\cal Z})\cap\mathbb{Z}^{d}$
and $P\supseteq(c_{d}(\textrm{conv}(c''Q')+\tilde{\cal Z}))\cap\mathbb{Z}^{d}$ for some $c_d > 0$. Since the volume of $\tilde{\cal Z}$ is at least $\tilde{c}'|A|^{d+1}$, the volume of $c_d(\textrm{conv}(c''Q')+\tilde{\mathcal{Z}})$ is at least $c_1'\max(|Q'|,\textrm{Vol}(\tilde{\cal Z})) \ge c_1'' \max(|Q'|,|A|^{d+1})$. 
Hence, by a variant of Minkowski's convex body theorem due to van der Corput~\cite{vdC} saying that a symmetric convex body in $\mathbb{R}^d$ of volume larger than $2^dk$ contains at least $k$ integer points, 
the number of integer points in the symmetric convex body $c_d(\textrm{conv}(c''Q')+\tilde{\mathcal{Z}})$ 
is at least $c_1'''\max(|Q'|,|A|^{d+1})$. Note that the affine span of $(c_{d}(\textrm{conv}(c''Q')+\tilde{\cal Z}))\cap \mathbb{Z}^{d}$ is $\mathbb{R}^d$ and $(c_{d}(\textrm{conv}(c''Q')+\tilde{\cal Z})) \cap \mathbb{Z}^{d}$ contains a translate of both $c_2'''Q'$ and $c_2'''\tilde{\mathcal{Z}} \cap \mathbb{Z}^d$. We thus have that $\Sigma(A)+Q'$ contains a translate of a GAP $P$ of size at least $c'''\max(|Q'|,|A|^{d+1})$ 
whose affine span is $\mathbb{R}^{d}$ and $P$ contains $c'''Q'$ and $c'''\tilde{\mathcal{Z}}\cap \mathbb{Z}^d$, where $c'''=\min(c_1''',c_2''')$. 
\end{proof}

\section{Proof of Theorem \ref{thm:hom-AP}}

We are now in a position to prove Theorem \ref{thm:hom-AP}. In this section, 
for simplicity of notation,
we will often use the same symbol $c_f$ for different constants that depend on a particular parameter $f$, but allowing the value to change from line to line. 

\begin{proof}[Proof of Theorem \ref{thm:hom-AP}]
Let $\beta = k$, let $\epsilon$ be a constant which is sufficiently small in terms of $\beta$ and let $s=\frac{m}{(\log m)^{2}}$. 
By Theorem \ref{thm:hom-AP-build2}, 
we can find a subset $\hat{A}$ of $A$ with $|\hat{A}| \ge c|A|$ which is $(\epsilon,\beta)$-resilient and strongly-$(s,\beta)$-stable, a centered GAP $P$ with dimension $d$ bounded in terms of $\beta$ such that
$\hat{A} \cup\{0\}$ is contained in $P$ and a subset $A'$
of $\hat{A}$ of size at most $s$ 
such that $\Sigma(A')$
contains a proper homogeneous translate
of $c_{\beta}sP$. 
Without loss of generality, we can assume that $P$ is symmetric, noting that, since $P$ is centered, this extends each of the widths of $P$ by at most a factor of $2$. Since $c_{\beta}sP$ is proper and $c_{\beta}sP$
is contained in a translate of $\Sigma(A')$, which is itself a subset of $[0,ns]$, we have 
\begin{equation}\label{eq:sizebound}
cc_d(c_\beta s)^d m \le c_{d}\left(c_{\beta}s\right)^{d}|P|\le\left|c_{\beta}sP\right|\le  ns+1 \le m^{\beta+1},
\end{equation}
where the first inequality follows since $P$ contains $\hat{A}\cup \{0\}$ and the second inequality follows from Lemma \ref{lem:GAP-size}. 
If $d>\beta$, we would then have that $d\ge \lfloor \beta \rfloor +1$, as $d$ is an integer. But, by (\ref{eq:sizebound}), this implies that $m^{d-\beta}\le (cc_d c_\beta^d)^{-1} (\log m)^{2d}$, which is false for $n$ sufficiently large. Therefore, $d \leq \beta = k$.

Suppose $P=\{\sum_{i=1}^{d} n_iq_i\,:\,n_i\in [a_i,b_i]\}$ with $b_i=-a_i$ (as we are assuming $P$ is symmetric) 
and $\phi$ is the identification map $\phi:c_\beta sP\to \mathbb{Z}^{d}$. 
Consider the map $\psi:\mathbb{Z}^{d}\to \mathbb{Z}$ given by $\psi(n_1,\dots,n_d) = \sum_{i=1}^{d}n_iq_i$. Since 
$\Sigma(A')$ contains a translate of $c_\beta sP$,  $\psi(\Sigma(\phi(\hat{A}\setminus A'))+\phi(c_\beta sP))$ is contained in a translate of $\Sigma(\hat{A})$. By Lemma~\ref{lem:GAP-size}, $|s(\hat{A}\cup \{0\})| \le s^{d}|P_d(\hat{A}\cup \{0\})|$ and, on the other hand, $|s(\hat A\cup \{0\})|\ge |\Sigma(A')|\ge |c_\beta sP| \gg_\beta s^{d}|P|$. Hence, we have $|\phi(P)| =|P| \ll_{\beta} |P_d(\hat A \cup \{0\})|$, so, provided 
$\epsilon$ is sufficiently small in terms of $\beta$, we may apply Corollary \ref{cor:dim-d} to conclude that, under the map $\phi$, any subset of $\hat{A}$ of size at least $|\hat{A}|/100$ has dimension $d$. 
Thus, we have that $\phi(\hat{A}\cup \{0\})$ is a subset of a box $\phi(P)$ in $\mathbb{Z}^d$ with widths $2b_1+1,\dots,2b_d+1$ and under $\phi$ the dimension of any subset of $\hat{A}\cup \{0\}$ with size at least $|\hat{A}|/100$ is $d$. Moreover, $\phi(c_\beta sP)$ is a box with widths at least $c_\beta sb_1,\dots,c_\beta sb_d$, where $c_\beta sb_i \ge 8\sqrt{d(|\hat{A}|+1)}(2b_i+1)$. 
Hence, by Lemma~\ref{lem:build-GAP}, 
we obtain that $\Sigma(\phi(\hat{A}\setminus A'))+\phi(c_\beta sP)$
contains a GAP $P'$ whose affine span is $\mathbb{R}^d$, whose size is at least $c_{d}\max(|A|^{d+1},|c_\beta sP|)$ (here we use that $c_\beta sP$ is proper, so that $|\phi(c_{\beta} sP)|=|c_{\beta} sP|$) and where $P'$ contains a translate of $\phi(c_{d,\beta}sP)$ and $c_{d}\mathcal{Z}\cap \mathbb{Z}^{d}$, where $\mathcal{Z}$ is a translate of the zonotope of a subset of $\phi(\hat{A}\setminus A')$ 
of size at least $|\hat{A}\setminus A'|-2^d$. 

Thus, $\psi(P')$ is a GAP of dimension $d$ in $\mathbb{Z}$ with volume at least $c_{d}|A|^{d+1}$ and size at least $|c_{d,\beta}sP|\gg_{d,\beta} s^dm$. 
Furthermore,
$\psi(P')$ is homogeneous, since $P'$ is contained in $\Sigma(\phi(\hat{A}\setminus A'))+\phi(c_\beta sP)$, $P'$ contains a translate of $\phi(c_{d,\beta}sP)$ and $\psi(\hat{A}\cup \{0\})\subseteq \psi(P)$ and, hence, $\gcd(P')=\gcd(P)\mid \gcd(\hat{A} \cup \{0\})$. 
If $\psi(P')$ is not proper,  Lemma
\ref{lem:non-proper} implies that either $\psi(P')$ contains a proper
homogeneous GAP of dimension at most $d$ and size at least $c_{d}|A|^{d+1}$ or $\psi(P')$
contains a homogeneous GAP of dimension at most $d-1$ and size at least $c_{d,\beta} s^d m$. 

First, consider the case where $\psi(P')$ contains a homogeneous GAP of dimension at most $d-1$ and size at least
$c_{d,\beta} s^d m$. 
By repeated further applications of Lemma~\ref{lem:non-proper}, we may conclude that, for some $d'\in [1,d-1]$, $\psi(P')$ contains a proper $d'$-dimensional homogeneous GAP of size at least $ c_{d,\beta}s^{d}m>m^{d+1}/(\log m)^{2d+1}>m^{d'+1}$, 
as required. Moreover, the same conclusion holds if $\psi(P')$ contains a proper homogeneous GAP of dimension at most $d-1$ and size at least $c_d |A|^{d+1}$. 

Finally, consider the case where $\psi(P')$ contains a proper $d$-dimensional
homogeneous GAP of size at least $c_{d}|A|^{d+1}$. We have $\psi(P')\subseteq [0,mn]$ as $\psi(P')$ is contained in $\Sigma(\hat{A})$, so 
\[
c_{d}|A|^{d+1}\le mn+1,
\]
which implies that $d<\beta$ if $m\ge C_{\beta}n^{1/\beta}$
for sufficiently large $C_{\beta}$. Since $d$ is an integer, $d \le \lceil \beta\rceil - 1 = k-1$. Thus, the conclusion of the theorem holds in all cases. 
\end{proof}

\noindent \textbf{Remark.}
We can guarantee that $\psi(P')$ contains either a proper  $d'$-dimensional homogeneous GAP of size at least $m^{d+1}/(\log m)^{2d+1}$ for some $d'<d$ or a proper $d$-dimensional homogeneous GAP of size at least $c_d |A|^{d+1}$ and minimum width at least $c_{d,\beta}|A|$. Indeed, recall that $\psi(P')$ contains a translate of $c_{d,\beta}sP$, where $c_{d,\beta} sP$ is proper, and $\psi(P')$ is contained in $mP$. In particular, $(c_{d,\beta}s/m) \psi(P')$ is proper. 
Furthermore, the minimum width of $(c_{d,\beta} s/m) \psi(P')$ is at least $c_{d,\beta} s/m \cdot c_{d,\beta}s > m/(\log m)^5$, where we use that $\psi(P')$ contains a translate of $c_{d,\beta} sP$ to conclude that it has minimum width at least $c_{d,\beta}s$. 
Therefore, in the proof of Lemma \ref{lem:non-proper}, we can check that only Case 1 can occur and, thus, either $\psi(P')$ contains a proper $d'$-dimensional homogeneous GAP of size at least $m^{d+1}/(\log m)^{2d+1}$ for some $d'<d$ or it contains a proper homogeneous translate of $c_d\psi(P')$. 

Hence, it remains to verify that the minimum width of $\psi(P')$ is at least $c_{d,\beta} |A|$, which implies that the minimum width of $c_d \psi(P')$ 
is at least $c_{d,\beta} |A|$. For this, we note that $P'$ contains $c_{d}\mathcal{Z}\cap \mathbb{Z}^{d}$, where $\mathcal{Z}$ is a translate of the zonotope of a subset $A_*$ of $\hat{A}\setminus A'$ of size at least $|\hat{A}\setminus A'|-2^d$. Write $P'=\{\sum_{i=1}^{d}n_i p_i:n_i \in I_i\}$ and assume, without loss of generality, that $|I_1|$ is the minimum of the $|I_i|$. Note that $p_1,\dots,p_d$ form a basis of $\mathbb{Z}^d$ and define projection maps $\pi_i:\mathbb{Z}^d\to \mathbb{Z}$ by $\pi_i(x)=n_i$ if $x=\sum_{i=1}^{d}n_ip_i$. Recall, from our application of Corollary \ref{cor:dim-d}, that any subset of $\hat{A}$ of size at least $|\hat{A}|/100$ has full dimension under $\phi$. 
Thus, $A_*$ contains at least $c_\beta |A|$ elements $x$ with $\pi_1(x)\ne 0$. We then obtain that $\max(\pi_1({\mathcal{Z}}))-\min(\pi_1({\mathcal{Z}})) \ge c_\beta |A|$ 
and, hence, since $P'$ contains $c_{d}\mathcal{Z}\cap \mathbb{Z}^d$, we have that $|I_1| \geq c_{d,\beta}|A|$, as required.

\section{Maximum non-averaging sets}

In this section, we prove Theorem \ref{thm:non-avg}, the main tool
being Theorem \ref{thm:hom-AP-build}. Let $\tilde{H}(n)$ be the maximum integer
for which there are two non-averaging subsets $A$ and $\tilde{A}$ of $[n]$ of size $\tilde{H}(n)$ with $\max(A)<\min(\tilde{A})$
whose sets of subset sums have no non-zero common element. 
As for the function $H(n)$ (see \cite[Corollary~1.10]{CFP} and its proof), we can show that 
\begin{equation}\label{usefulineq}
h(n)\le2\tilde{H}(n)+2.
\end{equation} 
Moreover, 
\begin{equation}\label{nextineq} 
\tilde{H}(n)\le H(n)\le Cn^{1/2}.
\end{equation} 

To prove Theorem \ref{thm:non-avg}, it thus suffices to prove the following result. 

\begin{thm}
There is an absolute constant $C$ such that, for all $n \geq 2$,
\begin{equation}\label{lastthm} 
\tilde{H}(n) \leq Cn^{\sqrt{2}-1}(\log n)^{2}.
\end{equation} 
\end{thm}

\begin{proof} We prove the theorem by strong induction on $n$. Let $n_0$ be any fixed positive integer. As $\tilde{H}(n) \leq n$ holds trivially, by taking $C$ sufficiently large, we may assume that (\ref{lastthm}) holds for all $2 \leq n \leq n_0$, giving us the base cases of our strong induction. For the induction hypothesis, assume that $n > n_0$ and (\ref{lastthm}) holds for all $n' < n$. Our aim for the rest of the proof is to show that (\ref{lastthm}) holds for $n$. 

Let $\alpha=\sqrt{2}-1$. Let $\tilde{H}(n)=m$ and assume, for the sake of contradiction, that $m>Cn^{\alpha}(\log n)^{2}$. Then there are non-averaging subsets $A$ and $\tilde{A}$ of $[n]$ of size $m$ with $\max(A)<\min(\tilde{A})$ whose sets of subset sums have no non-zero common element.

\begin{claim}\label{claim:intersect-Sigma}
If $\Sigma(A)$ contains a homogeneous progression $P$ of length larger than $n$, then $\Sigma(\tilde{A})$ must intersect $\Sigma(A)$ in a non-zero element. 
\end{claim}

{\noindent \it Proof of Claim.} 
Let $a$ be the common difference of $P$ and $x$ its initial element. By the pigeonhole principle, any set of $a$ integers contains a non-empty subset whose sum is divisible by $a$. We may therefore partition $\tilde{A}$ greedily into subsets $T_1\cup \dots\cup T_s$, each of size at most $a$, such that, for each $i\le s-1$, the sum of the elements in $T_i$ is a multiple of $a$. Furthermore, the sum of the elements in each $T_i$ is at most $an$. Thus, as long as $\sum_{z \in T_1\cup \dots \cup T_{s-1}} z > x$, $\Sigma(\tilde{A})$ intersects $P$ in a non-zero element. But if we let $M = \max(A) < \min(\tilde{A})$, then $x \le M|A| - an = Mm-an$, 
whereas $\sum_{z \in T_1\cup \dots \cup T_{s-1}} z > M|A'| - an = Mm-an$. Thus, $\Sigma(\tilde{A})$ intersects $\Sigma(A)$ in a non-zero element, as required. \qed

\vspace{3mm}

By Theorem \ref{thm:hom-AP-build}, there exists an absolute constant $c>0$ such that, for some $d$, there is a subset $\hat{A}$ of $A$ of size at least $c|A|$, a $d$-dimensional
GAP $P$ containing $\hat{A} \cup\{0\}$ and a subset $A'$ of $\hat{A}$ of size
at most $\frac{cm}{\log m}$ such that $\Sigma(A')$ contains a proper homogeneous translate
of $\frac{c^2m}{\log m}P$. 
Furthermore,
we have $d \le 2$. Indeed, if $d\ge 3$, then $$|\Sigma(A')| \ge \left(\frac{c^2m}{2\log m}\right)^3 |P| \ge \frac{c^7}{8} \frac{m^4}{(\log m)^3} > mn,$$
where we used that $|P| \ge |\hat{A}| \ge cm$, $m > Cn^{\alpha}(\log n)^{2}$ and, by Lemma~\ref{lem:GAP-size}, that $|kP| \ge (k/2)^d |P|$ whenever $P$ is a $d$-dimensional GAP and $kP$ is proper. 
However, this contradicts $\Sigma(A') \subseteq \Sigma(A) \subseteq [mn]$. 

We first consider the case $d=1$. Let $L$ be the length of $P$.
Since $P$ contains $\hat{A}$, which is a non-averaging set of size $cm$, we have $h(L)\ge cm$. As $h(L) = O(L^{1/2})$ by (\ref{usefulineq}) and (\ref{nextineq}), there is a constant $c_0>0$ such that $L\ge c_0m^{2}$. Thus, $\Sigma(\hat{A})$ contains a homogeneous progression 
of length at least $\frac{c^2m}{2\log m}L\ge\frac{c^2c_0m^{3}}{2\log m}>n$, where the last inequality holds as $C$ is sufficiently large and $m > Cn^{\alpha}(\log n)^{2}$, where $\alpha = \sqrt{2}-1 > 1/3$. By Claim \ref{claim:intersect-Sigma}, this is a contradiction.

Suppose now that $d = 2$. Let $P=x+[0,w_{1}-1]q_{1}+[0,w_{2}-1]q_{2}$
with $w_{1}\le w_{2}$. First, consider the case where $w_2 \ge n$. Since $\Sigma(A')$ contains a proper translate of $\frac{c^2m}{\log m}P$ for a subset $A'$ of $\hat{A}$ of size at most $\frac{cm}{\log m}$ and $\Sigma(A')$ is a subset of $[mn]$, we have 
\[
mn \ge \left|\frac{c^2m}{\log m}P\right| \ge \frac{c^4}{4} \frac{m^2}{(\log m)^2} w_1w_2 \ge \frac{c^4}{4} \frac{m^2}{(\log m)^2} n.
\]
Hence, $m/(\log m)^2 \le 4c^{-4}$, which contradicts our assumption that $m>Cn^{\alpha}(\log n)^2$ for a sufficiently large choice of $C$. 

Next, consider the case where $w_2 < n$. 
Since $\hat{A}$ is a non-averaging set, the intersection of $\hat{A}$ with each translate of $[0,w_{2}-1]q_{2}$ has size at most $h(w_{2})\le 3\tilde{H}(w_2) \leq 3Cw_2^{\alpha}(\log w_2)^{2}$, where the first inequality is by (\ref{usefulineq}) and the second inequality is by the induction hypothesis. Hence, the size of $\hat{A}$ is at most $w_1 h(w_2)$, implying that $w_{1}h(w_{2}) \ge |\hat{A}| \ge cm$, 
so $3Cw_1w_2^{\alpha}(\log w_2)^{2} \ge cm$. Since $w_2 \ge w_1$, we have $w_1w_2^{\alpha}(\log w_2)^{2} \le (w_1 w_2)^{(1+\alpha)/2}(\log (w_1w_2))^{2}$. Thus, 
\[
w_1w_2 \ge \frac{1}{4}(cm/3C)^{2/(1+\alpha)}/(\log(cm/3C))^{4/(1+\alpha)}.
\]
Hence, 
\[
\left|\frac{c^2m}{\log m}P\right|\ge \frac{c^4}{4} \frac{m^{2}}{(\log m)^{2}}w_{1}w_{2}\ge \frac{c^4}{4} \frac{m^{2}}{(\log m)^{2}} \cdot \frac{1}{4}(cm/3C)^{\sqrt{2}}/(\log(cm/3C))^{2\sqrt{2}}. 
\]

Since $\Sigma(A')$ is a proper subset of $[mn]$, we have 
\[
\frac{c^{4+\sqrt{2}}}{16(3C)^{\sqrt{2}}} \frac{m^{2+\sqrt{2}}}{(\log m)^{2} (\log (cm/3C))^{2\sqrt{2}}}\leq \left|\frac{c^2m}{\log m}P\right| \leq |\Sigma(A')|< mn \le \frac{1}{C^{1/\alpha}} \frac{m^{1+1/\alpha}}{(\log (m/C))^{2/\alpha}},
\]
where, in the last inequality, we used $m>Cn^{\alpha}(\log n)^{2}$, so that $n < (m/C)^{1/\alpha}/(\log (m/C))^{2/\alpha}$. 
In particular, since $1/\alpha = \sqrt{2} + 1$,
\[
(\log m)^2 (\log (cm/3C))^{2\sqrt{2}}/(\log(m/C))^{2\sqrt{2}+2} \ge c^{4+\sqrt{2}} C / (16\cdot 3^{\sqrt{2}}),
\]
so 
\begin{equation}\label{alastineq}
(\log m)^2/(\log(m/C))^{2} \ge c^{4+\sqrt{2}} C / (16\cdot 3^{\sqrt{2}}).
\end{equation}

Recall now that $m>Cn^{\alpha}(\log n)^{2}$. If $m \leq C^2$ and $n > n_0$ is sufficiently large, 
the left-hand side of (\ref{alastineq}) is at most $4(\log C)^2$ and otherwise the left-hand side of (\ref{alastineq}) is at most $4$. In either case, as $C$ can be taken sufficiently large, (\ref{alastineq}) cannot be satisfied, a contradiction. Hence, $m\le Cn^{\alpha}(\log n)^{2}$, completing the induction.
\end{proof}

\end{document}